\documentclass{article}
\usepackage{amsmath}
\usepackage[english]{babel}

\usepackage{amsthm}
\usepackage{amsfonts}
\usepackage{graphicx}
\usepackage{float}
\usepackage{bbm}

\newtheorem{theorem}{Theorem}[section]

\newtheorem{lemma}[theorem]{Lemma}
\newtheorem{remark}[theorem]{Remark}
\newtheorem{proposition}[theorem]{Proposition}

\title{On the percolative properties of the intersection of two independent interlacements}
\author{Zijie Zhuang \thanks{Schools of Mathematical Sciences, Peking University, Beijing, China, 100871, E-mail: zzj19981029@pku.edu.cn.}}
\date{October 2020}

\newcommand{\RN}[1]{%
  \textup{\uppercase\expandafter{\romannumeral#1}}%
}
\numberwithin{equation}{section}

\begin{document}

\maketitle

\begin{abstract}
    We prove the existence of non-trivial phase transitions for the intersection of two independent random interlacements and the complement of the intersection. Some asymptotic results about the phase curves are also obtained. Moreover, we show that at least one of these two sets percolates in high dimensions. 
\end{abstract}
\section{Introduction}

The model of random interlacements was first introduced by Sznitman in \cite{unique K} to clarify the local structure left by a simple random walk on a discrete torus running up to some time proportional to its volume. It has interesting percolative and geometric properties, and a lot of research has been done in this field, e.g., \cite{quenched invariance}, \cite{further invariance}, \cite{percolation for vacant set} and \cite{unique K}. 

More precisely, random interlacements are a Poisson point process whose ``points'' are doubly-infinite trajectories on $\mathbb{Z}^d$ ($d \geq 3$), with the intensity measure governed by a parameter $u>0$. We let $\mathcal{I}^u$ denote the set of vertices visited by at least one of these trajectories and call it the {\it interlacement set at level $u$}. We let $\mathcal{V}^u$ denote the complement of $\mathcal{I}^u$ and call it the {\it vacant set at level $u$}. We refer to Section \ref{section_2} for precise definitions. 

In this article, we will consider two independent interlacements $\mathcal{I}^{u_1}_1$, $\mathcal{I}^{u_2}_2$ with intensity parameters $u_1$, $u_2$, and their vacant sets $\mathcal{V}^{u_1}_1$, $\mathcal{V}^{u_2}_2$. Let $\mathcal{K}^{u_1,u_2}=\mathcal{I}^{u_1}_1 \cap \mathcal{I}^{u_2}_2$ be their intersection and $\mathcal{V}^{u_1,u_2}= \mathcal{V}^{u_1}_1\cup \mathcal{V}^{u_2}_2$ be the complement of the intersection. Superscripts will be omitted whenever no ambiguity arises. 

We now present our main results on the percolative properties of the intersection and its complement. First, both $\mathcal{K}$ and $\mathcal{V}$ have at most one infinite connected component and undergo a non-trivial phase transition in $u_1$ and $u_2$. We also obtain some results about the asymptotic behavior of the phase curves. The phase curve of $\mathcal{V}$ will tend to the lines $x=u^+$ and $y=u^+$, where $u^+$ is a parameter between the two percolative thresholds of interlacements $u_*$ and $u_{**}$ defined in Section \ref{section_2} (see Figure 1). The phase curve of $\mathcal{K}$ will tend to $x$-axis and also $y$-axis (see Figure 2).\footnote{Currently, we do not know whether this curve hits the coordinate axes for $d=3, 4$, see Claims $(2)$ and $(4)$ of Theorem \ref{1.2}.}

\begin{figure}[h]
\centering
\includegraphics[scale=0.73]{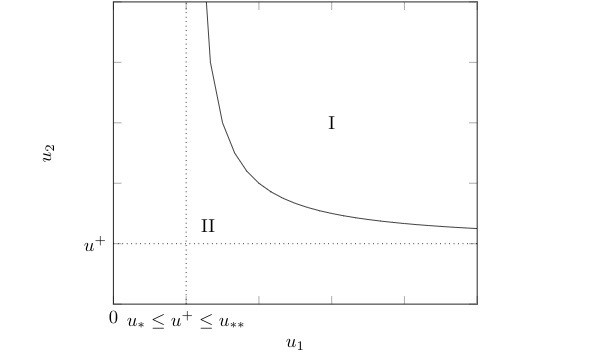}
\caption{(Phase diagram of $\mathcal{V}$.) Region \RN{1}: $\mathcal{V}$ does not percolate. Region \RN{2}: $\mathcal{V}$ percolates.}
\end{figure}

\begin{figure}[h]
\centering
\includegraphics[scale=0.73]{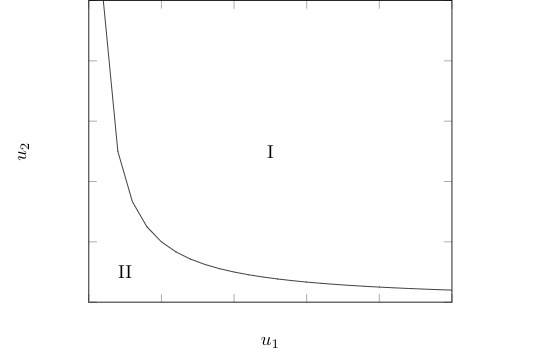}
\caption{(Phase diagram of $\mathcal{K}$.) Region \RN{1}: $\mathcal{K}$ percolates. Region \RN{2}: $\mathcal{K}$ does not percolate.}
\end{figure}

We also research the phase graph of $\mathcal{K}$ and $\mathcal{V}$ put together and consider the questions whether there is a phase where two infinite components coexist and whether there is a phase where neither of them exists. It follows from the above asymptotic analysis that there exists a certain region such that both $\mathcal{K}$ and $\mathcal{V}$ percolate. The second question is hard and depends on the dimension, e.g., Bernoulli site percolation on $\mathbb{Z}^d$. In low dimensions, it might be the case that the occupied vertices and vacant vertices wrap each other. We claim that in high dimensions at least one of $\mathcal{K}$ and $\mathcal{V}$ percolates through showing that the phase curve of $\mathcal{K}$ lies below $\{(x,y):x \geq 1,y \geq 1 \}$. In this case, the phase graph of $\mathcal{K}$ and $\mathcal{V}$ put together is as follows (Figure 3).

\begin{figure}[H]
\centering
\includegraphics[scale=0.5]{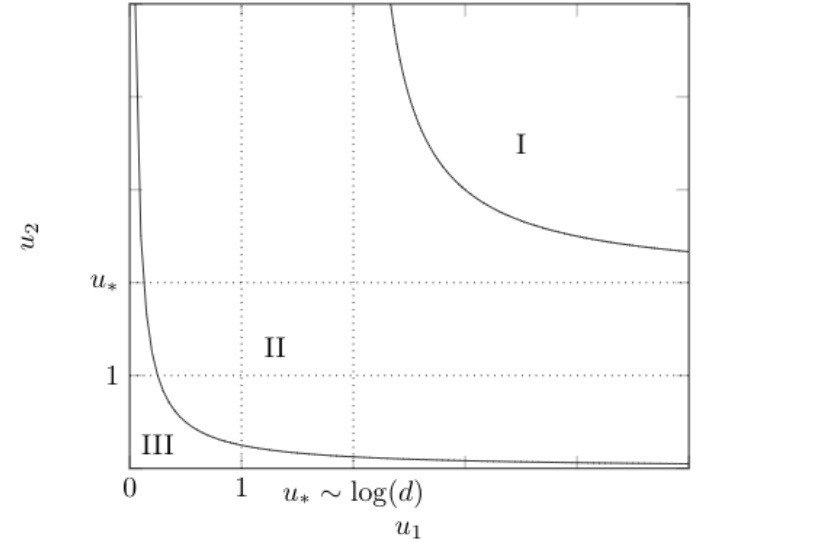}
\caption{(Phase diagram of $\mathcal{K}$ and $\mathcal{V}$ in high dimensions.) Region \RN{1}: $\mathcal{K}$ percolates but $\mathcal{V}$ does not percolate. Region \RN{2}: both $\mathcal{K}$ and $\mathcal{V}$ percolate. Region \RN{3}: $\mathcal{K}$ does not percolate but $\mathcal{V}$ percolates.}
\end{figure}

The motivation of this article comes from the study of random walks. This intersection is a good approximation of the intersection of two independent random walks on a torus running up to some time proportional to its volume, which has similar properties to the intersection of two independent simple random walks on $\mathbb{Z}^d$ conditional on intersecting many times \cite{large deviation intersect}, \cite{wiener sausages}. This work is the starting point of the study on the disconnection problem (see \cite{disconnection1}, \cite{disconnection2}) of the intersection of two independent random walks. 

Next, we will state our results rigorously and briefly explain the main ideas after each theorem. Let $B(x,r)$ be the $l^{\infty}$ ball centered at $x$ and of radius $r$. For a finite subset $K$ of $\mathbb{Z}^d$, let $\partial_i K$ be its inner boundary. First, we present the result about the percolative properties of $\mathcal{V}$.

\begin{theorem}[Percolative properties of $\mathcal{V}$]
\label{1.1}
\leavevmode
\begin{itemize}
\item[{\rm (1).}] The set $\mathcal{V}$ contains at most one infinite component a.s.
\item[{\rm (2).}]When $u_1<u_*$ or $u_2<u_*$, there is a.s.\ a unique infinite component in $\mathcal{V}$.
\item[{\rm (3).}] Given $u_1>u_{**}$, there exists $C=C(u_1,d)$ such that for all $u_2>C$, there are a.s.\ no infinite components in $\mathcal{V}$.
\item[{\rm (4).}] There exists a constant $u^+ \in [u_*, u_{**}]$ and a decreasing function $\Gamma:[u^+, + \infty) \rightarrow [u^+, + \infty]$ $($only $\Gamma(u^+)$ can be $+\infty)$ \footnote{Currently, we do not know whether $\Gamma$ is continuous.} such that $\mathcal{V}$ a.s.\ has a unique infinite component when \begin{itemize} 
\item[i.] $u_1<u^+$;
\item[ii.] $u_1\geq u^+$ and $u_2<\Gamma(u_1)$, \end{itemize}
and $\mathcal{V}$ a.s.\ has no infinite components when $u_1 \geq u^+$ and $u_2>\Gamma(u_1)$. 
\end{itemize}
\end{theorem}
Claim $(1)$ is an elementary property of most percolation models on $\mathbb{Z}^d$ or more generally an amenable graph. The proof of it uses a variant of the Burton-Keane argument \cite{BK}, \cite{unique V}. Claim $(2)$ is immediate from the definition of $u_*$. For $(3)$, one can see $\mathcal{V}^{u_2}_2$ as a small perturbation when $u_2$ is large. Thus, Claim $(3)$ mainly says that the percolation of the vacant set of interlacements is stable under this fluctuation. The proof relies on the renormalization argument introduced in \cite{percolation for vacant set} and local properties of random interlacements. The renormalization argument builds on an induction along the renormalization scheme and provides us with the decoupling inequalities (see Proposition \ref{decoupling}). Thus, we only need to prove the ``triggers", i.e., some local inequalities in a finite box $B(0, 2L_0)$. Locally, with high probability $\mathcal{V}$ cannot have a large connected component, since with high probability $\mathcal{V}^{u_1}_1 $ cannot connect $B(0,L_0)$ with $\partial_i B(0,2L_0)$ and $\mathcal{V}^{u_2}_2$ is empty in $B(0,2L_0)$ by enlarging $u_2$. Finally, combining $(2)$ and $(3)$, we can get $(4)$ instantaneously.

Next, we will present our result about the percolative properties of $\mathcal{K}$. These properties are different from those in the original model, i.e., random interlacements.

\begin{theorem}[Percolative properties of $\mathcal{K}$]
\label{1.2}
\leavevmode
\begin{itemize}
\item[{\rm (1).}] The set $\mathcal{K}$ has at most one infinite component a.s.

\item[{\rm (2).}] Given $u_1>0$, there exists a constant $C=C(u_1,d)<\infty$ such that $\mathcal{K}$ has a unique infinite component a.s.\ for all $u_2>C$.

\item[{\rm (3).}] There exist a constant $c=c(d)>0$ such that for all $u_1,u_2<c$, there are a.s.\ no infinite components in $\mathcal{K}$.

\item[{\rm (4).}] Given $d \geq 5$ and $u_1>0$, there exists a constant $c=c(u_1,d)>0$ such that for all $u_2<c$, there are a.s.\ no infinite components in $\mathcal{K}$.
\end{itemize}
\end{theorem}
 Claim (1) is elementary and its proof is the same as that of (1) in Theorem \ref{1.1}. For (2), one can see $\mathcal{V}^{u_2}_2$ as a small fluctuation when $u_2$ is large. Thus, Claim (2) mainly says that the percolation of the intersection $\mathcal{K}=\mathcal{I}^{u_1}_1 \backslash \mathcal{V}^{u_2}_2$ is stable under this fluctuation. The proof uses local properties of random interlacements and the renormalization argument. Locally, random interlacements are strongly connected meaning that with high probability all the vertices of $\mathcal{I}^u$ in $B(0,N)$ are connected in $B(0,2N)$. Meanwhile, $u_2$ can be taken large such that with high probability $\mathcal{V}^{u_2}_2$ is empty in $B(0,2N)$. Thus, with high probability, $\mathcal{K}$ has a large connected component in $B(0,2N)$. Then, through the renormalization argument, these local properties can be pushed to the global ones. The rigorous proof is a little bit harder since the above mentioned event isn't monotone and we cannot use the decoupling inequalities directly to it. For the result of (3), we consider the box of side length $2N$ and write $M=u_1N^{d-2}=u_2N^{d-2}$. First, pick a large $M$ to offset the error terms in the decoupling inequalities. In $\mathcal{I}^{u_1}_1 \cap B(x,2N)$ and $\mathcal{I}^{u_2}_2 \cap B(x,2N)$, with high probability, there are $O(M)$ independent random walks individually. Given $M$, one can take $N$ large and simultaneously $u_1$ and $u_2$ small such that only with small probability $\partial_i B(0,N)$ is connected to $\partial_i B(0,2N)$ by the intersection. Then, we can use the renormalization argument to push these local properties to the global ones. The rigorous proof will need some concrete calculations on simple random walks. Claim (4) is an improvement of Claim (3) for $d \geq 5$. Its proof uses cut times of random walks \cite{cut time}.

\medskip

By now, two natural questions arise: is there a phase where two infinite components coexist? Similarly, is there a phase where neither of them exists? Our results above also shed some light to these questions. By (2) in Theorem \ref{1.1} and (2) in Theorem \ref{1.2}, there exist choices of $u_1$ and $u_2$ such that both $\mathcal{K}$ and $\mathcal{V}$ have an infinite component. For the second question, we give an affirmative answer when the dimension is high.  Together with Theorem 0.1 in \cite{lower bound}, Claim (2) in Theorem \ref{1.1} and the following Theorem \ref{1.3}, we conclude that when the dimension is high, at least one of $\mathcal{K}$ and $\mathcal{V}$ has an infinite component. We summarize the discussion above into the following theorem.
\begin{theorem}
\label{1.4}
\leavevmode
\begin{itemize}
\item[{\rm (1).}] There exists a phase such that $\mathcal{K}$ and $\mathcal{V}$ both have a unique infinite component a.s.
\item[{\rm (2).}] There exists $D_2 <\infty$ such that for all $d> D_2$ and $u_1,u_2>0$, at least one of $\mathcal{K}$ and $\mathcal{V}$ has a unique infinite component a.s.
\end{itemize}
\end{theorem}
\noindent A key ingredient of the proof of (2) in Theorem \ref{1.4} is the following Theorem \ref{1.3}.
\begin{theorem}
\label{1.3}
There exists a constant $D_1$ such that for all $d>D_1$ and $u_1,u_2>1$, there is a.s.\ a unique infinite component in $\mathcal{K}$. In other words, the phase curve of $\mathcal{K}$ lies below the region $\{(x,y):x \geq 1\mbox{ and } y \geq1\}$.
\end{theorem}
As a remark, we note that the lower bound $1$ here is not optimal and can be improved to $\frac{1}{d^{1/2-\epsilon}}$. However, the argument cannot be extended to low dimensions as it involves some asymptotic analysis.

The proof of this theorem contains two parts: a local analysis and a renormalization argument similar to those in the proof of the theorems above. Locally, in a hypercube $\{0,1\}^d$, random interlacements can stochastically dominate a Bernoulli site percolation. Thus, by the method in \cite{high dimension}, in high dimensions, in each hypercube with high probability there is a ubiquitous component (see Section \ref{section_5} for the definition) connected to the neighboring ones (which also exist with high probability). Then, by the renormalization argument, there will be an infinite component.

Next, we will explain how this article is organized. In Section \ref{section_2}, we introduce some notations, make a brief introduction to random interlacements, explain the renormalization argument which we will use repeatedly in this article and give some estimates on simple random walks. Section \ref{section_3} is devoted to the proof of Theorem \ref{1.1}. In Section \ref{section_4}, we prove Theorem $\ref{1.2}$. The phase diagram of $\mathcal{K}$ and $\mathcal{V}$ put together is discussed in Section \ref{section_5}. 

Finally, we explain the convention regarding constants in this work. All constants in this article are positive. Constants like $ c,C, \epsilon, \gamma$ may change from place to place, while constants with subscripts like $c_1,C_1,D_1$ are kept fixed through the article. The constants in Section \ref{section_3} and \ref{section_4} may depend on $d$ implicitly, while the constants in Section \ref{section_5} will not.
\section{Notation and useful results}
\label{section_2}
In this section, we introduce notation, review some basic properties of random interlacements together with the renormalization argument introduced in Section 2 of \cite{percolation for vacant set} and collect some estimates related to simple random walks.

For a real value $a$, we write $[a]$ for the largest integer $\leq a$. We consider the integer lattice $\mathbb{Z}^d$ with $d \geq 3$. Norms $| \cdot |_1$ and $| \cdot |_{\infty}$ represent the $l^1$-norms and $l^{\infty}$-norms on $\mathbb{Z}^d$. We call two vertices $x$ and $y$ *-neighbors if $| x-y |_{\infty}=1$ and nearest neighbors if $| x-y |_1=1$. We call a set $\pi=(y_1,y_2,...,y_k) \subset \mathbb{Z}^d$ a *-neighbor path if $y_i,y_{i+1}$ are *-neighbors for $1 \leq i \leq k-1$, and a nearest neighbor path if $y_i,y_{i+1}$ are nearest neighbors for $1 \leq i \leq k-1$. A path is simple if $y_i$s are all different from each other. Given $K,L,U$ subsets of $\mathbb{Z}^d$, we say $K$ and $L$ are connected by $U$ and write $K \overset{U}{\longleftrightarrow}L$, if there exists a nearest neighbor path with values in $U$ which starts in $K$ and ends in $L$. We denote by $B(x,N)=\{ y \in \mathbb{Z}^d:|x-y|_{\infty}\leq N \}$ the closed $l^{\infty} $ ball centered at $x$ and of radius $N$. For a finite subset $K$ of $\mathbb{Z}^d$, we write $\partial_i K=\{x \in K: x \mbox{ is a nearest neighbor of some point }y \notin K \}$ for its inner boundary. 

Here is some notation about discrete-time simple random walks. $P_x$ represents the discrete-time simple random walk $X$ started at $x$ on $\mathbb{Z}^d$. We write $P_{x,y}$ for two independent discrete-time simple random walks started at $x$ and $y$ respectively. Let $K$ be a finite subset of $\mathbb{Z}^d$. We write $\tau_K$ for the first time that $X$ hits $K$ and $\tau_K^+$ for the first positive time that $X$ hits $K$. We denote the equilibrium measure of $K$ by $e_K(x)=P_x\left[\tau_K^+=\infty \right]\mathbbm{1}_K(x)$ for $x \in \mathbb{Z}^d$, and the capacity of $K$ by $\mbox{cap} (K)=\sum_x e_K(x)$.

\subsection{Random interlacements and the renormalization argument}
First, we briefly introduce the random interlacements.

Let $W$ be the space of doubly-infinite nearest neighbor paths in $\mathbb{Z}^d$, and let $W^*$ be the quotient space of $W$ modulo time shift. $\pi$ is the quotient map from $W$ to $W^*$. By Chapter 5 of \cite{introduction}, we can define a Poisson point measure $\mu$ on $W^*$ with the following local property. Given $K$ a finite subset of $\mathbb{Z}^d$, we write $W^*_K$ for the paths in $W^*$ that pass through $K$ and $\mu_K$ for $\mu$ restricted to $W^*_K$. Then
\begin{equation}
\label{def_RI}
    \mu_K=\sum_{i=1}^{N_K}\delta_{\pi(X_i)},
\end{equation}
where $N_K$ is a random variable $\sim$ ${\rm Poisson}(u\cdot  {\rm cap}(K))$ and $X_i$ is a doubly-infinite path in which $X_i(0)$ is a random point in $K$ according to the equilibrium measure distribution. Conditional on $X_i(0)$, the positive side $\{X_i\}_{i \geq 0}$ is a simple random walk, and the opposite side $\{X_i\}_{i \leq 0}$ is a simple random walk conditional on $\{ \tau_K^+=\infty \}$ and independent of the positive one. Given $N_K$, all these $N_K$ paths are conditionally independent. The set of points occupied by at least one path is called the interlacement set at level $u$, denoted by $\mathcal{I}^u$. The complement of it is called the vacant set, denoted by $\mathcal{V}^u$. The graph induced by the edges visited by random interlacements on $\mathbb{Z}^d$ is denoted by $\Tilde{\mathcal{I}}^u$.

There is a more concise alternative definition of interlacements. The random interlacements $\mathcal{I}^u$ are a random subset of $\mathbb{Z}^d$ whose law is given by
\begin{equation}
\label{def_RI_2}
    P\left[\mathcal{I}^u \cap K=\emptyset \right]=e^{-u \cdot {\rm cap}(K)}\mbox{, for all finite subset }K \mbox{ of }\mathbb{Z}^d.
\end{equation}

Let $(\Omega_1,\mathcal{F}_1,P^{u_1})$ be the probability space on which $\mathcal{I}^{u_1}_1$ is defined (see (5.2.1) and (5.2.6) of \cite{introduction} for more details). Let $(\Omega_2,\mathcal{F}_2,P^{u_2})$ be the probability space on which $\mathcal{I}^{u_2}_2$ is defined. Finally, let $(\Omega,\mathcal{F},\mathbbm{P})=(\Omega_1 \times \Omega_2,\sigma (\mathcal{F}_1 \times \mathcal{F}_2),P^{u_1} \otimes P^{u_2})$ be the probability space on which $\mathcal{I}^{u_1}_1$ and $\mathcal{I}^{u_2}_2$ are jointly defined.

Random interlacements are a typical model of percolation with long-range correlation. It has been known that the interlacement set itself almost surely has a unique infinite component as shown in (2.21) of \cite{unique K}, while the vacant set undergoes a non-trivial phase transition in $u$ (see Theorem 4.3 of \cite{unique K}, Theorem 3.4 of \cite{percolation for vacant set} and Theorem 3.1 of \cite{soft local times}). There are two percolative thresholds of interlacements $u_*$ and $u_{**}$ that we will use in this paper:
\begin{equation}
\label{def_u*}
    u_*=\inf \left\{ u:P\left[0\overset{\mathcal{V}^u}{\longleftrightarrow}\infty\right]=0 \right\},
\end{equation}
\begin{equation}
\label{def_u**}
    u_{**}=\inf \left\{ u:\liminf_{L\rightarrow \infty}P\left[B(0,L) \overset{\mathcal{V}^u}{\longleftrightarrow} \partial_i B(0,2L) \right] = 0 \right\}.
\end{equation}

When $u$ is above $u_*$, there are a.s.\ no infinite clusters in the vacant set. When $u$ is above $u_{**}$, each component of the vacant set is exponentially small. There is another critical parameter $\overline{u}$ introduced in Theorem 1.1 of \cite{local property}. It is plausible, but unproven at the moment, that actually $\overline{u}=u_*=u_{**}$, which is one of the most important open problems in this field. In the context of the level-set percolation of Gaussian free field, a model which bears similar properties to random interlacements, the parallel problem has been solved recently in Theorem 1.1 of \cite{equality of parameters}.

Next, we will state the renormalization argument first introduced in Chapter 2 of \cite{percolation for vacant set}. The idea is to zoom in on a large box layer upon layer along the renormalization scheme. In each layer, we can decouple the configurations in two far apart boxes but with small errors. The version we present here is from Chapter 8 of \cite{introduction}. Let $L_0$ and $l_0$ be two positive integers chosen according to the context, and $L_n=L_0 \cdot l_0^n$.
We define the renormalized lattice graph $\mathbb{G}_n$ as
\begin{equation*}
    \mathbb{G}_n=L_n\mathbb{Z}^d=\left\{ L_nx\mbox{ : }x \in \mathbb{Z}^d \right\}.
\end{equation*}
For $x\in \mathbb{Z}^d$ and $n \geq 0$, let
\begin{equation*}
    \Lambda_{x,n}=\mathbb{G}_{n-1} \cap B(x,L_n)
\end{equation*}
be a renormalized box with side length $L_n$. We call an event $G_x=G_{x,0}$ seed event if it is measurable with respect to the configuration in $B(x,2L_0)$ and shift-invariant, i.e., $\psi \in G_x$ if and only if $\psi(\cdot -x) \in G_0$, where $x$ is a vertex of $\mathbb{Z}^d$. We also hope $G_x$ monotone. For $x\in \mathbb{Z}^d$ and $n \geq 1$, we write
\begin{equation}
\label{seed}
    G_{x,n}=\underset{x_1,x_2 \in \Lambda_{x,n};|x_1-x_2|_{\infty}>\frac{L_n}{100}}{\bigcup}G_{x_1,n-1}\cap G_{x_2,n-1}.
\end{equation}
$G_{x,n}$ means that there exists a dyadic tree whose leaves are separated apart and satisfy $G_x$.

A typical scenario where we will use this event is in the following claim. When there is a *-neighbor or nearest neighbor path connecting $0$ to $\partial_i B(0,L_n)$ such that every vertex of this path satisfies $G_x$, then $G_{x,n}$ happens. This claim can be proved by induction and it is used in Propositions \ref{nonpercolate_V}, \ref{percolate_K} and \ref{nonpercolate_K}. 

The following decoupling inequalities are a variant of Theorem 8.5 of \cite{introduction} and they are used repeatedly in our proofs. They follow from the idea of renormalization and the sprinkling technique in Proposition 3.1 of \cite{unique K}.
\begin{proposition}[Decoupling inequalities for two interlacements]
\label{decoupling}
For $d\geq3$ and $\epsilon>0$, there exists an integer $A=A(d,\epsilon)$ such that for all $n \geq 0$, $L_0 \geq 1$ and $l_0 \geq A$, the following two statements hold:

1.\ if $G_x$ is an increasing seed event, then for all ${u_1}^- \leq (1-\epsilon)u_1$ and  ${u_2}^- \leq (1-\epsilon)u_2$
\begin{equation}
\begin{split}
\label{decoup_increase}
    &\mathbbm{P}\left[\mathcal{K}^{{u_1}^-,{u_2}^-} \in G_{0,n}\right] \\
    \leq &(2l_0+1)^{d \cdot 2^{n+1}}\left[\mathbbm{P}\left[\mathcal{K}^{u_1,u_2} \in G_0\right] + \epsilon ({u_1}^-, L_0, l_0)+ \epsilon ({u_2}^-, L_0, l_0)\right] ^{2^n};
\end{split}
\end{equation}

2.\ if $G_x$ is a decreasing seed event, then for all ${u_1}^+ \geq (1+\epsilon)u_1$ and  ${u_2}^+ \geq (1+\epsilon)u_2$
\begin{equation}
\begin{split}
\label{decoup_decrease}
   & \mathbbm{P}\left[\mathcal{K}^{{u_1}^+,{u_2}^+} \in G_{0,n}\right]\\
    \leq &(2l_0+1)^{d \cdot 2^{n+1}}\left[ \mathbbm{P}[\mathcal{K}^{u_1,u_2} \in G_0]+ \epsilon ({u_1}, L_0, l_0)+ \epsilon ({u_2}, L_0, l_0)\right]^{2^n},
\end{split}
\end{equation}
where
\begin{equation}
\label{def_epsilon}
    \epsilon (u, L_0, l_0)= \frac{2e^{-uL_0^{d-2}l_0^{\frac{d-2}{2}}}}{1-e^{-uL_0^{d-2}l_0^{\frac{d-2}{2}}}}.
\end{equation}
\end{proposition}
\begin{proof}
The proof is similar to Theorem 8.5 in \cite{introduction} despite that the claim here involves two independent interlacements. We need to change the coupling in Theorem 7.9 of \cite{introduction} to a coupling of two independent copies of the point measures there and the term $\epsilon$ there should be changed to $\epsilon_1+\epsilon_2$, where $\epsilon_1 = \epsilon({u_1}_-,{u_1}_+,S_1,S_2,U_1,U_2)$ and $\epsilon_2=\epsilon({u_2}_-,{u_2}_+,S_1,S_2,U_1,U_2)$. The term $\epsilon(u_-,n)$ in (8.1.9) and (8.1.10) of \cite{introduction} should be replaced by $\epsilon({u_1}_-,n)+\epsilon({u_2}_-,n)$. Equation (8.3.4) of \cite{introduction} is still true since $a^m+b^m+c^m \leq (a+b+c)^m$.
\end{proof}

\subsection{Estimates about SRW}
Here, we present some results about simple random walks on $\mathbb{Z}^d$ that will be used. For $x,y \in \mathbb{Z}^d$, the Green function is denoted by $G(x,y)=\sum_{n=0}^{\infty}P_x(X_n=y)$. The following lemma is very simple and we give a proof just for completeness. It is used in Section \ref{section_5} to prove that interlacements can dominate Bernoulli site percolation in a hypercube $\{0,1 \}^d$.
\begin{lemma}
\label{dominate}
There exists a constant $c_1>0$ such that for all $d \geq 3$,
\begin{equation}
\label{sum of green function}
    P_0\left[{\tau}_{\{0,1\} ^d}^+=\infty\right]>c_1.
\end{equation}
\end{lemma}
\begin{proof}
Thanks to symmetry and the strong Markov property, 
\begin{equation*}
    \sum_{x \in \{0,1\} ^d}G(0,x)= E_0\left[\sum_{i=0}^{\infty} \mathbbm{1}_{\left\{ X_i \in \{ 0,1 \} ^d \right\}} \right]=\frac{1}{ P_0\left[{\tau}_{\{ 0,1 \} ^d}^+=\infty \right]}.
\end{equation*}
Thus, it is sufficient to prove that
\begin{equation*}
    \sum_{x \in \{0,1\} ^d}G(0,x) < \frac{1}{c_1}.
\end{equation*}
By (2.10), p.243 in \cite{green function}, we have
\begin{equation*}
    G(0,x)=\int_{0}^{\infty}e^{-u} \prod_{i=1}^{d}I_{x_i}\left(\frac{u}{d}\right)du, \mbox{ for }x=(x_1,...,x_d) \in \mathbb{Z}^d,
\end{equation*}
where
\begin{equation*}
    I_n(u)=\frac{1}{\pi}\int_{0}^{\pi}e^{u\cos{\theta}}\cos{n\theta}d\theta, \mbox{ }u \in \mathbb{C}.
\end{equation*}
We get that
\begin{align*}
\sum_{x \in \{ 0,1\} ^d}G(0,x) &= \int_{0}^{\infty}e^{-u}\sum_{x \in \{ 0,1\} ^d} \prod_{i=1}^{d}I_{x_i}\left(\frac{u}{d}\right)du\\
&= d\int_{0}^{\infty}\left(\frac{I_0(u)+I_1(u)}{e^u}\right)^ddu.
\end{align*}
Denote $({I_0(u)+I_1(u)})/{e^u}$ by $Z(u)$. Then, $Z(0)=1$. By some easy calculations, we can take a constant $A \in (1,\infty)$ and $B=1/(4e^A) + 1/(4e^{2A})$ such that 
\begin{equation*}
    Z(u) \leq \frac{1}{\sqrt{u}} \mbox{, for all }u\geq A,
\end{equation*}
and
\begin{equation*}
    Z(u) \leq 1-Bu\mbox{, for all }u \leq A.
\end{equation*}
Therefore,
\begin{align*}
\sum_{x \in \{ 0,1 \} ^d}G(0,x) &= d\int_{0}^{A}(Z(u))^ddu + d\int_{A}^{\infty}(Z(u))^ddu \\
& \leq d\int_{0}^{A}(1-Bu)^ddu+d\int_{A}^{\infty}\left(\frac{1}{\sqrt{u}}\right)^ddu \\
& \leq \frac{1}{B}\frac{d}{d+1} +A^{-\frac{d}{2}+1}\frac{d}{\frac{d}{2}-1}.
\end{align*}
Note that $d \geq 3$, $A >1$ and $A, B$ are independent of $d$. Hence, there exists a constant $c_1>0$ independent of $d$ such that $\sum_{x \in \{ 0,1\} ^d}G(0,x) < 1/c_1$. This completes the proof of \eqref{sum of green function}.
\end{proof}
\begin{remark}
With more careful calculations, we can get that the left-hand side in \eqref{sum of green function} tends to 1/2 as $d$ tends to $\infty$. This coincides with the heuristic that whenever $X_i$ leaves $\{ 0,1\} ^d$ in high dimensions, it will not come back any more.
\end{remark}
The following notation is defined when $d=3$. Let $x=(x_1,x_2,x_3) \in \mathbb{Z}^3$. The disc centered at $x$ of radius $M$ is denoted by 
\begin{equation*}
    D(x,M)=\{ (y_1,y_2,y_3) : |y_1-x_1|,|y_2-x_2| \leq M, y_3=x_3 \}.
\end{equation*} We denote one quarter of the disc centered at $x$ of radius $M$ by 
\begin{equation*}
    D^+(x,M)=\{ (y_1,y_2,y_3) : x_1 \leq y_1 \leq x_1+M, x_2 \leq y_2 \leq x_2+M, y_3=x_3\}.
\end{equation*} The first exit time of a simple random walk is written by $\xi_m= \inf \{ i \geq 0: X_i \notin B(0,m) \}$. The following proposition is used to prove Claim (3) of Theorem \ref{1.2}, or equivalently Proposition \ref{nonpercolate_K}.
\begin{proposition}
\label{key_estimate}
For $d=3$, there exists $C_1>0$ such that for all $M \geq 1$,
\begin{equation}
\label{key_estimate_eq}
    \max_{x,y \in \partial_i B(0,2M)}P_{x,y}\left[X^1[0,\infty) \cap X^2[0,\infty) \cap \partial_i B(0,M) \neq \emptyset \right] < \frac{C_1}{\log (M)},
\end{equation}
where $X^1$ and $X^2$ are two independent simple random walks starting from $x$ and $y$.
\end{proposition}
The order $\Omega(1/ \log(M))$ here is right. The proof is similar to Section 3.4 in \cite{intersections of random walks}. The expectation of the intersection in $\partial_i B(0,M)$ is $O(1)$. Intuitively, when $X^1$ and $X^2$ intersect in $\partial_i B(0,M)$, then they will have $O(\log(M))$ intersection points in $\partial_i B(0,M)$. Thus, the probability that $X^1$ and $X^2$ intersect in  $\partial_i B(0,M)$ is $\Omega(1 /\log(M))$. However, there are no natural stopping times, which makes the rigorous proof difficult. To prove this proposition, we will need two lemmas.
\begin{lemma}
\label{2.4}
For $d=3$, there exists $\epsilon>0$ and a positive integer $N$ such that for all $M \geq N$,
\begin{equation*}
    \min_{x \in B(0,M+1)}P_x\left[\sum_{i=0}^{\xi_{2M}-1} \mathbbm{1}_{ \left\{X_i \in D^+(0,2M) \right\}}> \epsilon M\right]> \epsilon.
\end{equation*}
\begin{proof}
Write $y$ for $(3M/2, 3M/2, 0)$. By comparing the simple random walk with Brownian motion, there exists $\mu>0$ such that for $M$ large
\begin{equation}
\label{2.4_1}
     \min_{x \in B(0,M+1)}P_x\left[\tau_{D(y,\frac{1}{4}M)}<\xi_{2M}\right]>\mu.
\end{equation}
Furthermore, we can prove the following inequality with some $\gamma>0$ and large $M$:
\begin{equation}
\label{2.4_2}
    P_0\left[\sum_{i=0}^{\xi_{M/4}} \mathbbm{1}_{  \left\{X_i \in D(0,\frac{1}{4}M) \right\}}> \gamma M\right] >\gamma.
\end{equation}
This inequality can be proved by considering the third coordinate. The movements in the third coordinate $x_3$ can be seen as a one-dimensional simple random walk $\{ Y_i \}_{i \geq 0}$ starting from $0$. By just calculating the first moment and second moment (first moment is of order $\sqrt{n}$ and second moment is of order $n$) and then using the Paley-Zygmund inequality, we can prove that there exists some $\eta>0$ such that for large $n$,
\begin{equation}
\label{2.4_2_1}
    P_0\left[\sum_{i=0}^{n} \mathbbm{1}_{\{Y_i=0\}} > \eta \sqrt{n}\right] > \eta.
\end{equation}
In addition, we can take some $c>0$ such that for large M,
\begin{equation}
\label{2.4_2_2}
P_0[Z]>1-\frac{\eta}{2},
\end{equation}
where $Z$ represents the event that $\xi_{M/4} >cM^2$ and in the first $cM^2$ moves there are at least $cM^2/4$ ones in the third coordinate.

If the event in \eqref{2.4_2_1} with $n=[cM^2/4]$ and $Z$ happen at the same time, then the event in \eqref{2.4_2} happens with $\gamma \leq  \eta \sqrt{c}/2 $. Note that with more than $\eta /2$ probability the event in \eqref{2.4_2_1} with $n=[cM^2/4]$ and $Z$ happen at the same time. Let $\gamma = \min \{ \eta /2, \eta \sqrt{c}/2  \}$ and this completes the proof of \eqref{2.4_2}.

By \eqref{2.4_1}, \eqref{2.4_2} and the strong Markov property, for any $x \in B(0,M+1)$ and large $M$,
\begin{align*}
&P_x\left[\sum_{i=0}^{\xi_{2M}-1} \mathbbm{1}_{ \left\{X_i \in D^+(0,2M) \right\}}> \gamma M\right] \\
\geq &P_x\left[\tau_{D(y,\frac{1}{4}M)}<\xi_{2M}\right] \cdot P_0\left[\sum_{i=0}^{\xi_{M/4}} \mathbbm{1}_{ \left\{X_i \in D(0,M/4) \right\}}> \gamma M\right] > \mu \cdot \gamma.
\end{align*}
In Lemma \ref{2.4}, take $\epsilon= \min \{ \gamma, \mu \gamma \}$ and $N$ large. We complete the proof.
\end{proof}
\end{lemma}
We consider the random variable $D_M$ defined by
\begin{equation*}
    D_M=\sum_{i=0}^{\infty}G\left(0,X_i \right) \mathbbm{1}_{\left\{ X_i \in D^+(0,M) \right\}}.
\end{equation*}
\begin{lemma}
\label{2.5}
For $d=3$, there exist two constants $a,\epsilon > 0$ and and a positive integer $N$ such that for all $M \geq N$
\begin{equation*}
P_0\left[D_M \geq \epsilon \log (M)\right] \geq 1-\frac{1}{M^a}.
\end{equation*}
\end{lemma}
\begin{proof}
Denote the $\epsilon$ and $N$ in Lemma \ref{2.4} by $\epsilon_1$ and $N_1$. We decompose $\left[0,\xi_{M}\right)$ into $k$ disjoint intervals $\left[\xi_{N_1}, \xi_{2N_1}\right) , \left[\xi_{2N_1}, \xi_{4N_1}\right), ..., \left[\xi_{2^{k-1}N_1}, \xi_{2^kN_1}\right)$, where $k=\left[\log_2(M/N_1)\right]$. Recall that $G(0,x) \geq C/|x|_{\infty}$ in $d=3$. Therefore, 
\begin{align*}
    D_M &=\sum_{i=0}^{\infty}G\left(0,X_i^1\right) \mathbbm{1}_{\left\{ X_i^1 \in D^+(0,M) \right\}}\\
    &\geq \sum_{l=0}^{k-1}\sum_{i=\xi_{2^lN_1}}^{\xi_{2^{l+1}N_1}-1}G\left(0,X_i^1\right) \mathbbm{1}_{\left\{ X_i^1 \in D^+(0,2^{l+1}N_1) \right\}}\\
    &\geq  \sum_{l=0}^{k-1}{\left( \sum_{i=\xi_{2^lN_1}}^{\xi_{2^{l+1}N_1}-1}\frac{C}{2^{l+1}N_1} \mathbbm{1}_{\left\{ X_i^1 \in D^+(0,2^{l+1}N_1) \right\} } \right)}.
\end{align*}
By the strong Markov property and Lemma \ref{2.4}, each of the $k$ terms above is independent and has more than $\epsilon_1$ probability to be more than $C/(2^{l+1}N_1) \cdot \epsilon_1 2^l N_1= C \epsilon_1 /2$. Thus, by the Hoeffding's inequality, for large $k$
\begin{equation*}
    P_0\left[D_M \geq \frac{C \epsilon_1^2}{4}k \right] \geq 1-e^{-ck}.
\end{equation*}
In Lemma \ref{2.5}, take $\epsilon <\frac{C\epsilon_1^2}{4}\log_2e$, $a < c\log_2e$ and $N$ large. We complete the proof.
\end{proof}
With the above two lemmas, we can complete the proof of Proposition \ref{key_estimate}.
\begin{proof}[Proof of Proposition \ref{key_estimate}]
We will use the $a$ and $\epsilon$ in Lemma \ref{2.5}. Take two points $x$ and $y$ in $\partial_i B(0,2M)$. The constants below are all independent of $x, y$ and $M$. Now, $X^1$ and $X^2$ below are two independent simple random walks started at $x$ and $y$. Define $R_M$ to be
\begin{equation*}
    R_M=\sum_{i=0}^{\infty}\sum_{j=0}^{\infty} \mathbbm{1}_{\left\{ X_i^1=X_j^2 \in \partial_i B(0,M) \right\}}.
\end{equation*}
An easy calculation shows that 
\begin{equation*}
    E_{x,y}\left[R_M\right] \leq C \mbox{, where $C$ is a constant independent of }x,y \mbox{ and }M.
\end{equation*}
Let $\zeta$ be the stopping time 
\begin{equation*}
    \zeta =\inf \left\{ i \geq 0 : X^1_i \in X^2[0,\infty) \cap \partial_i B(0,M) \right\}.
\end{equation*}
Define $\sigma$ as
\begin{equation*}
    \sigma = \inf \left\{ j \geq 0: X^1_{\zeta}=X^2_j \right\}.
\end{equation*}
For any vertex $z$ in $\partial_i B(0,M)$, we can find a $M \times M$ disc in $\partial_i B(0,M)$ with $z$ as one of its corners in a deterministic way. We write $D(z)$ for this disc. The time $j$ is called good if
\begin{equation*}
    D_{j,M}= \sum_{i=0}^{\infty}G \left(X_j^2,X_{j+i}^2\right) \mathbbm{1}_{\left\{ X_{j+i}^2 \in D(X_j^2) \right\} } \geq \epsilon \log(M)
\end{equation*}
and bad otherwise. By the strong Markov property applied to $X^1$,
\begin{equation*}
    E_{x,y}\left[R_M|\zeta < \infty , \sigma \mbox{ is good}\right] \geq \epsilon \log(M).
\end{equation*}
Therefore,
\begin{align*}
    P_{x,y}\left[\zeta < \infty , \sigma \mbox{ is good}\right] &\leq E_{x,y}\left[R_M\right] \left[E_{x,y} \left[R_M|\zeta < \infty , \sigma \mbox{ is good}\right]\right]^{-1}\\
    &\leq \frac{C}{\log(M)}.
\end{align*}
By Lemma \ref{2.5}, symmetry and the Markov property applied to $X^2$, we have $P_{x,y}[D_{j,M} < \epsilon \log (M)] \leq 1/M^a$ for large $M$. Thus, for large $M$
\begin{align*}
    P_{x,y}[\zeta <\infty, \sigma \mbox{ is bad}] &\leq
    \sum_{i=0}^{\infty}\sum_{j=0}^{\infty}P_{x,y}\left[X^1_i=X^2_j \in \partial_i B(0,M),D_{j,M}< \epsilon \log (M)\right]\\
    &=
    \sum_{i=0}^{\infty}\sum_{j=0}^{\infty}P_{x,y}\left[X^1_i=X^2_j \in \partial_i B(0,M)\right]P_{x,y}\left[D_{j,M} < \epsilon \log (M)\right]\\
    &\leq \frac{1}{M^a} E_{x,y}\left[R_M\right] \leq \frac{C}{M^a}.
\end{align*}
So,
\begin{equation*}
\begin{split}
    &P_{x,y}\left[X^1[0,\infty) \cap X^2[0,\infty) \cap \partial_i B(0,M) \neq \emptyset \right] =P_{x,y}[\zeta <\infty] \\
    =& P_{x,y}[\zeta <\infty, \sigma \mbox{ is good}]+ P_{x,y}[\zeta <\infty, \sigma \mbox{ is bad}] \leq \frac{C}{\log (M)}.
\end{split}
\end{equation*}
The above inequality holds when $M$ is large. Enlarge $C$ if necessary such that for any $M$, inequality \eqref{key_estimate_eq} holds.
\end{proof}
\section{Percolative properties of $\mathcal{V}$}
\label{section_3}
In this section, we consider the percolative properties of $\mathcal{V}$ and prove Theorem \ref{1.1} which is split into four parts, namely four propositions below.
\begin{proposition}
\label{unique_V}
There exists at most one infinite component in $\mathcal{V}$ a.s.
\end{proposition}
\begin{proposition}
\label{existence_V}
If $u_1<u_*$ or $u_2<u_*$, then $\mathcal{V}$ a.s.\ has a unique infinite component.
\end{proposition}
\begin{proposition}
\label{nonpercolate_V}
Given $u_1>u_{**}$ there exists $C=C(u_1)$ such that for all $u_2>C$, there are a.s.\ no infinite components in $\mathcal{V}$.
\end{proposition}
\begin{proposition}
\label{phase curve}
There exist a constant $u^+ \in [u_* ,u_{**}]$ and a decreasing function $\Gamma:[u^+, + \infty) \rightarrow [u^+, + \infty]$ $($only $\Gamma(u^+)$ can be $+\infty)$ such that $\mathcal{V}$ a.s.\ has a unique infinite component when \begin{itemize} \item[i.] $u_1<u^+$; \item[ii.] $u_1\geq u^+$ and $u_2<\Gamma(u_1)$,\end{itemize} and $\mathcal{V}$ a.s.\ has no infinite components when $u_1 \geq u^+$ and $u_2>\Gamma(u_1)$.
\end{proposition}

The proof of Proposition \ref{unique_V} is an adaptation of the Burton-Keane argument (see Theorem 2 of \cite{BK}, Corollary 2.3 of \cite{unique K} and Theorem 1.1 of \cite{unique V}). The proof presented here is a streamlined version of that in Theorem 1.1 of \cite{unique V}. Also, the maps $\phi$ and $\Phi$ defined below are adapted from (2.23) to (2.24) in \cite{unique K} and (3.2), (3.4), (3.11) in \cite{unique V}. The idea is to change the situation in a finite box by local surgeries.

\begin{proof}[Proof of Proposition \ref{unique_V}]
Recall that random interlacements are translation-invariant and ergodic as shown in Theorem 2.1 of \cite{unique K}. It follows that the total number $N$ of infinite connected components of $\mathcal{V}$ is a.s.\ a constant, possibly infinite. The proof contains two parts. The first step is to argue that
\begin{equation*}
    \mbox{for }1 < k < \infty, \mbox{ }\mathbbm{P}[N=k]=0.
\end{equation*}
Suppose that in contrast for some $k \in (1, \infty)$, there are $k$ infinite clusters a.s. Then, there exists a large constant $M$ such that $\mathbbm{P}[A]>0$, where $A$ denotes the event that all the $k$ infinite connected components of $\mathcal{V}$ intersect the box $K=[-M,M]^d$.

We first prove that $\mathbbm{P}[A_1]=0$, where $A_1$ denotes the event that $\mathcal{V} \backslash K$ contains more than $k$ infinite connected components. We consider the following map $\phi_1:W^* \rightarrow W^*$. Recall that $W^*_K$ is the set of the paths in $W^*$ that intersect $K$. For a path $\omega$ in $W^*_K$ , whenever $\omega$ enters $K$, the map $\phi_1$ adds to $\omega$ a subpath that covers $K$ and leaves $K$ at that point. For the paths in $(W^*_K)^c$, the map $\phi_1$ is an identity map. $\phi_1$ can induce a natural map $\Phi_1 : \Omega \rightarrow \Omega$, where $\Omega$ is the configuration space of two independent interlacements defined in Section \ref{section_2}. The map $\Phi_1$ is defined as:
\begin{equation*}
    \Phi_1(\psi)=\sum_{i,j} \delta_{(\phi (\omega_i), \phi (\omega'_j))}\mbox{, for }\psi=\sum_{i,j} \delta_{(\omega_i,\omega'_j)} \in \Omega.
\end{equation*}
Note that $\Phi_1(A_1) \subset \{ N>k \}$ (because the situation outside $K$ is unchanged and $K$ is occupied) and $\Phi_1 \circ \mathbbm{P} $ is absolutely continuous with respect to $\mathbbm{P}$. Thus,
\begin{align*}
    \mathbbm{P}[A_1] &\leq \mathbbm{P}[\Phi_1^{-1}\{ N>k \}] \\
    &= \Phi_1 \circ \mathbbm{P}[N>k]=0.
\end{align*}

Therefore $\mathbbm{P}[A \backslash A_1]=\mathbbm{P}[A]>0$. If the event $A$ happens, then we can find two vertices $z_1$ and $z_2$ in $\partial_i K$ such that they are vacant and contained in two distinct infinite components of $\mathcal{V}$. Since the number of choices is finite, there exist $z_1$ and $z_2$ in $\partial_i K$ such that $\mathbbm{P}[B]>0$, where $B$ represents the event that $\{ A \backslash A_1,\mbox{vertices }z_1 \mbox{ and }z_2\mbox{ satisfy the above conditions} \}$. Choose a set $U \subset K$ containing a simple path joining $z_1$ and $z_2$. We also demand that $U \cap \partial_i K=\{z_1,z_2 \}$. Consider the following map $\phi_2:W^* \rightarrow W^*$. For a path $\omega$ in $W^*_K$, whenever $\omega$ enters $K$, the map $\phi_2$ replaces the subpath of $\omega$ until it leaves $K$ by a subpath that bypasses $U$ and comes out of $K$ at the same point as $\omega$ (If $\omega$ enters or leaves $K$ at $z_1$, then the subpath need not bypass $z_1$, the same for $z_2$). For the paths in $(W^*_K)^c$, the map $\phi_2$ is an identity map. Then we can define $\Phi_2$ induced by $\phi_2$ as before. Note that any path will not pass $U \backslash \{z_1,z_2 \}$ under $\phi_2$ and all the paths passing through $z_1$ or $z_2$ must pass $z_1$ or $z_2$ in the preimage of $\phi_2$. In addition, the situation outside $K$ remains unchanged. Thus, $\Phi_2 (B) \subset \{ N<k \}$, since $\mathcal{V} \backslash K$ contains at most $k$ infinite components and two of them are connected under the map $\Phi_2$. So, 
\begin{align*}
    \mathbbm{P}[B] &\leq \mathbbm{P}\left[\Phi_2^{-1}\{ N<k \}\right] \\
    &= \Phi_2 \circ \mathbbm{P}[N<k]=0.
\end{align*}
We get a contradiction. Hence, for $1 < k < \infty$, $\mathbbm{P}[N=k]=0$.

The second step is to reject that $\mathbbm{P}[N=\infty]>0$. We assume that on the contrary this happens. Then there exists $M$ such that $\mathbbm{P}[C]>0$, where $C$ denotes the event that at least $100^d$ distinct infinite components in $\mathcal{V}$ intersect $K=[-M,M]^d$. We can find three vacant vertices $y_1$,$y_2$ and $y_3$ in $\partial_i K$ such that they are at least distance 10 from each other and all the corners of the box, and belong to three distinct infinite components. Since there are finitely many choices, $\mathbbm{P}[D]>0$ for some $y_1$,$y_2$ and $y_3$ in $\partial_i K$, where $D$ represents the event that $C$ happens and $y_1,y_2,y_3$ satisfy the above conditions. We can find a subset $U \subset K$ such that (1). $U \backslash \{ 0 \}$ contains three disjoint simple paths from 0 to $y_1, y_2$ and $y_3$; (2). $U \cap \partial_i K= \{y_1,y_2, y_3 \}$.  We consider the following map $\phi_3:W^* \rightarrow W^*$. For a path $\omega$ in $W^*_K$, whenever $\omega$ enters $K$, the map $\phi_3$ replaces the subpath of $\omega$ until it leaves $K$ with a subpath that bypasses $U$, fills $K \backslash U$ and comes out of $K$ at the same point as $\omega$ (If $\omega$ enters or leaves $K$ at $y_1$, then the subpath need not bypass $y_1$, the same for $y_2$ and $y_3$). For the paths in $(W^*_K)^c$, the map $\phi_3$ is an identity map. We can define $\Phi_3$ induced by $\phi_3$ as before. Note that any path will not pass $U \backslash \{y_1,y_2,y_3 \}$ under $\phi_3$ and all the paths passing through $y_1, y_2$ or $y_3$ must pass it in the preimage. Besides, $K \backslash U$ is occupied and the situation outside $K$ has not been changed. Thus, in $\Phi_3(D)$, the vertex 0 is a trifurcation point meaning that 0 belongs to an infinite component of $\mathcal{V}$ which is split into three distinct components by deleting 0. By definition, $ \Phi_3 \circ \mathbbm{P} $ is absolutely continuous with respect to $ \mathbbm{P}$. This together with $0<\mathbbm{P}[D] \leq \Phi_3 \circ \mathbbm{P}[ \Phi_3(D)]$ implies that $\mathbbm{P}[\Phi_3(D) ]>0$. Thus, $\mathbbm{P}[0 \mbox{ is a trifurcation point}] \geq \mathbbm{P}[\Phi_3(D) ]>0$. This is impossible due to the Burton-Keane argument as shown in Theorem 2 of \cite{BK}.

In conclusion, either $N=0$ a.s.\ or $N=1$ a.s.
\end{proof}

\begin{proof}[Proof of Proposition \ref{existence_V}]
This is obvious since $\mathcal{V}=\mathcal{V}^{u_1}_1 \cup \mathcal{V}^{u_2}_2$ and $\mathcal{V}^u$ a.s.\ has an infinite cluster when $u<u_*$ by \eqref{def_u*}.
\end{proof}

\begin{proof}[Proof of Proposition \ref{nonpercolate_V}]
Take $h= (u_1+u_{**})/2>u_{**}$. We define our seed event $G_x$ as the event that $B(x,L_0)$ is connected with $\partial_i B(x,2L_0)$ in $\mathcal{V}$, where $L_0$ is an integer to be determined later. Then, $G_x$ is measurable with respect to the configuration in $B(x,2L_0)$, shift-invariant and decreasing. $G_x$ is contained in the union of the following two events: (1). in $\mathcal{V}^{u_1}_1$, the box $B(x,L_0)$ is connected to $\partial_i B(x,2L_0)$; (2). $\mathcal{V}^{u_2}_2 \cap B(x,2L_0) \neq \emptyset$. Since $h>u_{**}$, by Theorem 3.1 of \cite{soft local times} there exist constants $c,C>0$ depending on $h$ such that 
\begin{equation*}
    \mathbbm{P}\left[B(x,L_0) \overset{\mathcal{V}^h_1}{\longleftrightarrow} \partial_i B(x,2L_0)\right] \leq {C}e^{-{L_0}^{c}}.
\end{equation*}
Thus, by the above inequality and \eqref{def_RI_2},
\begin{align*}
\mathbbm{P}\left[\mathcal{K}^{h,u_2} \in G_x\right] &\leq \mathbbm{P}\left[B(x,L_0) \overset{\mathcal{V}^h_1}{\longleftrightarrow} \partial_i B(x,2L_0)\right] +\mathbbm{P}\left[\mathcal{V}^{u_2}_2 \cap B(x,2L_0) \neq \emptyset\right] \\
&\leq {C}e^{-{L_0}^{c}} + \sum_{y \in B(x,2L_0)} \mathbbm{P}\left[y \in \mathcal{V}^{u_2}_2\right] \\
&= {C}e^{-{L_0}^{c}} + (4L_0+1)^de^{-u_2 \cdot {\rm cap}(0)}.
\end{align*}
Next, we will use the decoupling inequalities, i.e., Proposition \ref{decoupling}. In Proposition \ref{decoupling}, take $\epsilon>0$ such that $u_1=(1+\epsilon)h$. In \eqref{decoup_decrease}, we take $u_1=h, u_1^+=u_1, u_2=g, u_2^+=(1+\epsilon)g$, where $g$ is a constant to be determined later. Therefore, for $l_0 \geq A$ ($A=A(d,\epsilon)$ is the integer in Proposition \ref{decoupling}),
\begin{equation}
\begin{split}
\label{3.11}
&\mathbbm{P}\left[\mathcal{K}^{u_1,(1+\epsilon)g} \in G_{0,n}\right]\\
\leq &(2l_0+1)^{d \cdot 2^{n+1}}\left[ \mathbbm{P}\left[\mathcal{K}^{h,g} \in G_0\right] + \epsilon (h, L_0, l_0) + \epsilon (g, L_0, l_0) \right]^{2^n} \\
\leq &(2l_0+1)^{d \cdot 2^{n+1}}\left[ Ce^{-{L_0}^{c}} + {(4L_0+1)}^d \cdot e^{-g \cdot {\rm cap} (0)}+\epsilon (h, L_0, l_0)+ \epsilon ({g}, L_0, l_0) \right]^{2^n}.   
\end{split}
\end{equation}
Recall that by \eqref{def_epsilon},
\begin{equation*}
    \epsilon (u, L_0, l_0)= \frac{2e^{-uL_0^{d-2}l_0^{\frac{d-2}{2}}}}{1-e^{-uL_0^{d-2}l_0^{\frac{d-2}{2}}}}.
\end{equation*}
Take $l_0=A$. There exists an integer $B$ such that for $L_0=g=B$ the right-hand side of \eqref{3.11} is smaller than $2^{-2^n}$, i.e.,
\begin{equation*}
    (2A+1)^{2d}\left(Ce^{-{B}^{c}} + {(4B+1)}^d \cdot e^{-B \cdot \rm{cap}(0)}+ \epsilon (h, B, A)+ \epsilon (B, B, A)\right) < \frac{1}{2}.
\end{equation*}

We claim that if 0 is connected to $\partial_i B(0,2L_n)$ in $\mathcal{V}$, then the event $G_{0,n}$ happens (see \eqref{seed} for the definition of $G_{0,n}$). We can prove this claim by induction. For $n=0$, this holds immediately. If for $n=k$ it holds, we consider the case $n=k+1$. For a simple path connecting $0$ to $\partial_i B(0,2L_{k+1})$, it must pass $\partial_i B(0,L_{k+1}/3)$ and $\partial_i B(0,2L_{k+1}/3)$. By the induction hypothesis, we can prove that the $L_k$ boxes first passed by the path in $\partial_i B(0,L_{k+1}/3)$ and $\partial_i  B(0,2L_{k+1}/3)$ satisfying $G_{x,k}$ and their distance is by definition larger than $L_{k+1}/100$. Hence, the claim holds for $n=k+1$. By induction, it holds for all $n$. Recall that $\mathcal{V}$ is decreasing in $u_2$. Thus, for $L_0=B$ and $q \geq (1+\epsilon)B$,
\begin{align*}
    \mathbbm{P}\left[0 \overset{\mathcal{V}^{u_1,q}}{\longleftrightarrow}\partial_i B(0,2L_n) \right]& \leq \mathbbm{P}\left[0 \overset{\mathcal{V}^{u_1,(1+\epsilon)B}}{\longleftrightarrow}\partial_i B(0,2L_n) \right] \\
    &\leq \mathbbm{P}\left[\mathcal{K}^{u_1,(1+\epsilon)B} \in G_{0,n}\right]\leq 2^{-2^n}.
\end{align*} 
As $n$ tends to $\infty$, the right-hand side tends to zero. Thus,
\begin{equation*}
    \mathbbm{P}\left[0 \overset{\mathcal{V}^{u_1,q}}{\longleftrightarrow}\infty \right]=0 \mbox{, for all }q \geq (1+\epsilon)B.
\end{equation*}
Take $C(u_1)=(1+\epsilon)B$ in Proposition \ref{nonpercolate_V} and the proof is completed.
\end{proof}
\begin{remark}
\label{exponential decay}
Although the connectivity function we obtain here has stretched exponential decay, by the method in Section 7 of \cite{soft local times} we can greatly improve the bound to exponential decay in $d \geq 4$ and exponential decay with a logarithmic correction in $d=3$. In other words, for $d \geq 4$, $u_1>u_{**}$ and $u_2>C(u_1)$, there exist two constants $c=c(u_1,u_2,d)$ and $C=C(u_1,u_2,d)$ such that
\begin{equation*}
    \mathbbm{P}\left[0 \overset{\mathcal{V}}{\longleftrightarrow} x\right] \leq C e^{-c |x|_1}.
\end{equation*}
For $d=3$, $u_1>u_{**}$, $u_2>C(u_1)$ and any $\epsilon>0$, there exist two constants $c=c(u_1,u_2,d,\epsilon)$ and $C=C(u_1,u_2,d,\epsilon)$ such that 
\begin{equation*}
    \mathbbm{P}\left[0 \overset{\mathcal{V}}{\longleftrightarrow} x\right] \leq C e^{-c \frac{|x|_1}{\log^{3+\epsilon}|x|_1}}.
\end{equation*}
The proof is similar to Section 7 in \cite{soft local times} despite that we have to change both $u_1$ and $u_2$ and produce two error terms in the sprinkling process. For the sake of brevity, we will not give a complete proof here.
\end{remark}
\begin{proof}[Proof of Proposition \ref{phase curve}]
Take $\Gamma(x)= \inf \{ u: \mathbbm{P}\left[0 \overset{\mathcal{V}^{x,u}}{\longleftrightarrow} \infty \right] =0 \}$. By monotonicity, $\Gamma(x)$ is a decreasing function. Let $u^+= \inf \{ x: \Gamma(x) <\infty \}$. By Propositions \ref{existence_V} and \ref{nonpercolate_V}, we have $u_* \leq u^+ \leq u_{**}$. It follows from the symmetry in $u_1$ and $u_2$ that $\Gamma(x) \geq u^+$ when $x \geq u^+$. This completes the proof.
\end{proof}
\begin{remark}
Here, $\Gamma$ is the phase curve of $\mathcal{V}$. We conjecture that it is continuous and strictly decreasing. A much more difficult problem is what happens to $\mathcal{V}$ on this curve, since we do not even know whether $u_*=u_{**}$ and what happens to $\mathcal{V}^{u_*}$.
\end{remark}

\section{Percolative properties of $\mathcal{K}$}
\label{section_4}
In this section, we study the percolative properties of $\mathcal{K}$ and prove Theorem \ref{1.2}. The original model, random interlacements, is almost surely connected and contains a unique infinite component as shown in Corollary 2.3 of \cite{unique K}. However, the intersection of two independent random interlacements is not connected (this can be easily proved) and may even have no infinite components. Theorem \ref{1.2} is split into four parts, namely the four propositions below.
\begin{proposition}
\label{unique_K}
There a.s.\ exists at most one infinite cluster in $\mathcal{K}$.
\end{proposition}
\begin{proposition}
\label{percolate_K}
Given $u_1>0$, there exists a constant $C=C(u_1)$ such that $\mathcal{K}$ a.s.\ has a unique infinite component for $u_2>C$.
\end{proposition}
\begin{proposition}
\label{nonpercolate_K}
There exists a constant $c=c(d)$ such that for all $u_1,u_2<c$, there are a.s.\ no infinite components in $\mathcal{K}$.
\end{proposition}
\begin{proposition}
\label{larger than 5}
Given $d \geq 5$ and $u_1>0$, there exists a constant $c=c(u_1,d)>0$ such that for all $u_2<c$, there are a.s.\ no infinite components in $\mathcal{K}$.
\end{proposition}
\begin{proof}[Proof of Proposition \ref{unique_K}]
We use the same method as in the proof of Proposition \ref{unique_V}. We denote by $N$ the number of infinite components. Then, $N$ is a constant a.s., possibly infinite. The proof contains two parts.

The first step is to prove that $\mathbbm{P}[2 \leq N <\infty]=0$. We assume that on the contrary there exists an integer $k \in (1, +\infty)$ such that $\mathbbm{P}[N=k]=1$. Then, there exists $M$ such that $\mathbbm{P}[A]>0$, where $A$ denotes the event that all the $k$ infinite components of $\mathcal{K}$ intersect $K=[-M,M]^d$. We introduce a map $\phi_4:W^* \rightarrow W^*$, which is an identity map on $(W^*_K)^c$. For $\omega \in W^*_K$, the map $\phi_4$ adds a subpath that fills the box $K$ when the first time $\omega$ enters $K$. The map $\phi_4$ induces a map $\Phi_4:\Omega \rightarrow \Omega$ and $\Phi_4 \circ \mathbbm{P}$ is absolutely continuous with respect to $\mathbbm{P}$. Note that in $\Phi_4(A)$ there is exactly one infinite component in $\mathcal{K}$, since the situation outside $K$ is unchanged and all the points in $K$ are occupied. Thus, $\mathbbm{P} [N=1] \geq \mathbbm{P} [\Phi_4(A)] >0$, which is a contradiction. Therefore, $\mathbbm{P} [2 \leq N <\infty]=0$. 

The second step is to prove that $\mathbbm{P}[N=\infty]=0$. We assume that in contrast $\mathbbm{P}[N=\infty]=1$. We can find $K=[-M,M]^d$ and $x_1,x_2,x_3 \in \partial_i K$ such that $\mathbbm{P}[B]>0$ where $B$ represents the event that $K$ intersects three distinct infinite components of $\mathcal{K}$ at $x_1,x_2$ and $x_3$. We consider the maps $\phi_4$ and $\Phi_4$ as before. Then, $\mathbbm{P}[\Phi_4(B)] >0$. In $\Phi_4(B)$, the vertex $0$ is a $M$-trifurcation point (meaning that $B(0,M)$ intersects an infinite component and this infinite component is split into at least three disjoint clusters if we close all the vertices in $B(0,M)$). Thus, the vertex $0$ has a positive probability to be a $M$-trifurcation point. By the Burton-Keane argument, this is impossible. Hence, $\mathbbm{P}[N = \infty]=0$.

In conclusion, either $N$ equals 0 a.s.\ or $N$ equals 1 a.s.
\end{proof}
The proof of Proposition \ref{percolate_K} is similar to that of Theorem 1 in \cite{quenched noise}. We will use the strong connectivity property of random interlacements as shown in Lemma 3.1 of \cite{quenched noise}, which says that all the occupied vertices in a finite box are connected in a slightly larger box with high probability. The idea of this proposition is that in a $[-N,N]^d$ box, every component has larger than $1-\exp(-N^\delta)$ probability to have capacity larger than $N^{(d-2)(1-\delta)}$. Thus, the simple random walks started from them have larger than $1-\exp(-N^\epsilon)$ probability to intersect in a slightly larger box $[-(1+\epsilon)N,(1+\epsilon)N]^d$. The rigorous statement is as follows.

\begin{proposition}
\label{strong_c}
Let $d\geq3$, $\epsilon>0$ and $u>0$. There exist constants $c=c(d,u, \epsilon)>0$ and $C=C(d,u,\epsilon)>0$ such that for all $R \geq 1$,
\begin{equation}
    P\left[\bigcap_{x,y \in \mathcal{I}^u \cap B(0,R)}{x \longleftrightarrow y\mbox{ in }\Tilde{\mathcal{I}}^u} \cap B(0,(1+\epsilon)R)\right] \geq 1-C\exp(-cR^{1/6}),
\end{equation}
where $\Tilde{\mathcal{I}}^u$ denotes the graph induced by the edges visited by random interlacements on $\mathbb{Z}^d$.
\end{proposition}
To prove Proposition \ref{percolate_K}, we first need to define seed events and prove estimates about them. Unfortunately, the above event is not monotone. Thus, it cannot be defined as the seed event directly. We will separate it into two monotone events. There will be three seed events: $\overline{E}_x,\overline{F}_x$ and $\overline{G}_x$. When we choose appropriate constants, each of them has small probability to happen.

$E_x$ is defined as the intersection of the following two events: $(1).$ for all $e \in \{0,1\}^d$, the graph $\Tilde{\mathcal{I}}^{u_1}_1 \cap (x+eL_0+[0,L_0)^d)$ contains a connected component with at least $3m({u_1}){L_0}^d/4$ vertices, where $m({u_1})=\mathbbm{P}[0 \in \mathcal{I}^{u_1}_1]=1-e^{-{u_1} \cdot \rm{cap} (0)}$; (2).\ all of the above $2^d$ components are connected in the graph $\Tilde{\mathcal{I}}^{u_1} \cap (x+[0,2L_0)^d)$. The event $F_x$ is defined as: for all $e \in \{0,1\}^d$, ${\mathcal{I}}^{u_1}_1 \cap (x+eL_0+[0,L_0)^d)$ contains at most $5m({u_1}){L_0}^d/4$ vertices. $G_x$ is defined as: $\mathcal{V}^{u_2}_2 \cap (x+[0,2L_0)^d)  = \emptyset$. The events $E_x$ and $G_x$ are increasing and $F_x$ is decreasing. Thus, the complements $\overline{E}_x$, $\overline{G}_x$ are decreasing and $\overline{F}_x$ is increasing. Furthermore, $\overline{E}_x, \overline{G}_x$ and $\overline{F}_x$ are measurable with respect to the configuration in $B(x,2L_0)$ and shift-invariant. Let $\overline{E}_x, \overline{G}_x, \overline{F}_x$ be our seed events. We need the following proposition to prove Proposition \ref{percolate_K}.

\begin{proposition}
\label{seed estimate} 
Given $u_1>0$, there exist integers $l_0,L_0$ and a constant $C=C(u_1)>0$ such that for all $u_2>C$,
\begin{equation}
\label{84}
    \mathbbm{P}\left[\mathcal{K}^{u_1,u_2} \in \overline{E}_{0,n}\right],\mathbbm{P}\left[\mathcal{K}^{u_1,u_2} \in \overline{F}_{0,n}\right],\mathbbm{P}\left[\mathcal{K}^{u_1,u_2} \in \overline{G}_{0,n}\right] \leq 2^{-2^n}.
\end{equation}
\end{proposition}
\noindent See \eqref{seed} for the definition of $\overline{E}_{0,n}$. One note is that here $\overline{E}_{0,n}$ is the hierarchical event defined by $\overline{E}_x$, not the complement of $E_{0,n}$.

\begin{proof}[Proof of Proposition \ref{seed estimate}]
Take $\epsilon=1/2$ in Proposition \ref{decoupling}. Let $l_0=A$ ($A$ is the integer in Proposition \ref{decoupling}). We begin with $\overline{E}_{0,n}$. By an appropriate ergodic theorem, e.g., Theorem 8.6.9.\ in \cite{ergodic},
\begin{equation}
\label{ergodic}
    \lim_{L_0 \rightarrow \infty}\left[\left|\frac{1}{L_0^d} | \mathcal{I}^u \cap [0,L_0)^d |-m(u)\right|> \delta \right] = 0, \mbox{ for all } \delta>0.
\end{equation}

By Proposition \ref{strong_c} and \eqref{ergodic}, we can prove that for fixed $u_1$ and $u_2$, 
\begin{equation}
\label{85}
    \lim_{L_0 \rightarrow \infty}\mathbbm{P}\left[\mathcal{K}^{u_1,u_2} \in \overline{E}_x\right]=0.
\end{equation}
For more details, one can see Lemma 4.2 in \cite{quenched noise}. Inserting $u_1^+=u_1, u_2^+=u_2, u_1=2u_1/3, u_2=2u_2/3$ into \eqref{decoup_decrease}, we get that
\begin{equation*}
\begin{split}
    &\mathbbm{P}\left[\mathcal{K}^{u_1,u_2} \in \overline{E}_{0,n}\right] \\
    \leq &(2A+1)^{d \cdot 2^{n+1}}\left[\mathbbm{P}\left[\mathcal{K}^{\frac{2}{3}u_1,\frac{2}{3}u_2} \in \overline{E}_0\right]+ \epsilon \left(\frac{2}{3}u_1, L_0, A\right)+ \epsilon \left(\frac{2}{3}u_2, L_0, A\right) \right]^{2^n}.
\end{split}
\end{equation*}
By \eqref{def_epsilon} and \eqref{85}, there exists $\Gamma=\Gamma(u_1,u_2)>0$ such that for $L_0>\Gamma$,
\begin{equation*}
    (2A+1)^{2d}\left(\mathbbm{P}\left[\mathcal{K}^{\frac{2}{3}u_1,\frac{2}{3}u_2} \in \overline{E}_0\right]+ \epsilon \left(\frac{2}{3}u_1, L_0, A\right)+ \epsilon \left(\frac{2}{3}u_2, L_0, A\right)\right) < \frac{1}{2}.
\end{equation*}
Combining the above two inequalities, we have for $L_0>\Gamma$,
\begin{equation*}
    \mathbbm{P}\left[\mathcal{K}^{u_1,u_2} \in \overline{E}_{0,n}\right] \leq 2^{-2^n}.
\end{equation*}
Note that $\overline{E}_{0,n}$ is independent of $\mathcal{I}^{u_2}_2$. Thus, we can let $\Gamma$ independent of $u_2$.

For $\overline{F}_{0,n}$, by \eqref{ergodic}, $\lim_{L_0 \rightarrow \infty}\mathbbm{P}\left[ \mathcal{K}^{u_1,u_2} \in \overline{F}_x\right]=0$. Inserting $u_1^-=u_1,u_2^-=u_2,u_1=2u_1,u_2=2u_2$ into \eqref{decoup_increase}, we get that 
\begin{equation*}
\begin{split}
   & \mathbbm{P}\left[\mathcal{K}^{u_1,u_2} \in \overline{F}_{0,n}\right] \\
    \leq &(2A+1)^{d \cdot 2^{n+1}}\left[\mathbbm{P}\left[\mathcal{K}^{2u_1,2u_2} \in \overline{F}_0\right]+ \epsilon (u_1, L_0, A)+ \epsilon (u_2, L_0, A) \right]^{2^n}.
\end{split}
\end{equation*}
Thus, we can take $\Lambda=\Lambda(u_1,u_2)$ such that for $L_0 \geq \Lambda$,
\begin{equation*}
    \mathbbm{P}\left[\mathcal{K}^{u_1,u_2} \in \overline{F}_{0,n}\right] \leq 2^{-2^n}.
\end{equation*}
Note that $\overline{F}_{0,n}$ is independent of $\mathcal{I}^{u_2}_2$. So, we can let $\Lambda$ independent of $u_2$.

For $\overline{G}_{0,n}$, by \eqref{def_RI_2},
\begin{align*}
    \mathbbm{P}\left[\mathcal{K}^{u_1,u_2} \in \overline{G}_x\right] &= \mathbbm{P}\left[\mathcal{V}^{u_2}_2 \cap (x+[0,2L_0)^d) \neq \emptyset\right]\\
    &\leq \sum_{y \in (x+[0,2L_0)^d)} \mathbbm{P}\left[y \in \mathcal{V}^{u_2}_2\right] = (2L_0)^d e^{-u_2 \cdot \rm{cap} (0)}.
\end{align*}
Inserting $u_1^+=u_1, u_2^+=u_2, u_1=2u_1/3, u_2=2u_2/3$ into \eqref{decoup_decrease}, we get that
\begin{equation}
\label{haha3}
\begin{split}
&\mathbbm{P}\left[\mathcal{K}^{u_1,u_2} \in \overline{G}_{0,n}\right] \\
\leq &(2A+1)^{d \cdot 2^{n+1}} \left[\mathbbm{P}\left[\mathcal{K}^{\frac{2}{3}u_1,\frac{2}{3}u_2} \in \overline{G}_0\right]+ \epsilon \left(\frac{2}{3}u_1, L_0, A\right)+ \epsilon \left(\frac{2}{3}u_2, L_0, A\right) \right]^{2^n}\\
\leq &(2A+1)^{d \cdot 2^{n+1}} \left[(2L_0)^d e^{-\frac{2}{3}u_2 \cdot \rm{cap} (0)}+ \epsilon \left(\frac{2}{3}u_1, L_0, A\right)+ \epsilon \left(\frac{2}{3}u_2, L_0, A\right) \right]^{2^n}.
\end{split}
\end{equation}
There exists an integer $H>\min \{ \Gamma, \Lambda \}$ such that 
\begin{equation*}
    (2A+1)^{2d}\left((2H)^d e^{-\frac{2}{3}H \cdot \rm{cap} (0)}+ \epsilon \left(\frac{2}{3}u_1, H, A\right)+ \epsilon \left(\frac{2}{3}H, H, A\right)\right) < \frac{1}{2}.
\end{equation*}
Take $L_0=H$ and $C(u_1)=H$ in Proposition \ref{seed estimate}. Note that the left-hand side of \eqref{haha3} is decreasing in $u_2$ when $L_0$ is fixed. Thus, for $L_0=H$ and $u_2>C(u_1)$, we have $P\left[\mathcal{K}^{u_1,u_2} \in \overline{G}_{0,n}\right] \leq 2^{-2^n}$. This completes the proof of \eqref{84}.
\end{proof}
\begin{proof}[Proof of Proposition \ref{percolate_K}]
Take the $l_0$, $L_0$ and $C(u_1)$ in Proposition \ref{seed estimate}. We call a vertex $x$ good if $E_x \cap F_x \cap G_x$ occurs, otherwise bad. We claim that if there exists a nearest neighbor path of good vertices with infinite length, then there exists an infinite component of $\mathcal{K}$ along this path. By $G_x$, along this path $\mathcal{V}^{u_2}_2$ is empty. It follows from $E_x$ and $F_x$ that there is only one component with more than $3m(u_1)L_0^d/4$ vertices in each $L_0$ box along the path and all these components are connected with the neighboring ones. Therefore, there is an infinite component in $\mathcal{K}$. If there is no nearest neighbor paths of good vertices with infinite length, then by the dual argument there are infinitely many *-neighbor circuits of bad vertices surrounding $0$.  

By an induction argument (see Lemma 5.2 of \cite{quenched noise}),
\begin{equation*}
\begin{split}
&\mathbbm{P}\left[x \mbox{ is connected to } \partial_i B(x,2L_n) \mbox{ by a *-neighbor path of bad vertices}\right]\\
\leq &4^d \left[\mathbbm{P}\left[\overline{E}_{0,n}\right]+\mathbbm{P}\left[\overline{F}_{0,n}\right]+\mathbbm{P}\left[\overline{G}_{0,n}\right] \right] \leq 12^d  \cdot 2^{-2^n}.
\end{split}
\end{equation*}
Therefore,
\begin{equation*}
    \sum_{x \in \mathbb{Z}^2 \times \{0 \}^{d-2}}\mathbbm{P}\left[Z_x\right] < \infty,
\end{equation*}
where $Z_x$ represents the event that $x$ is passed by a *-neighbor circuit of bad vertices surrounding $0$. Hence, there are finitely many such circuits. Therefore, an infinitely long path of good vertices exists a.s., which means that $\mathcal{K}$ percolates a.s.\ when $u_2>C(u_1)$.
\end{proof}
\begin{proof}[Proof of Proposition \ref{nonpercolate_K}]
Given $K$ a finite subset of $\mathbb{Z}^d$, we write $N^1_K$ and $N^2_K$ for the number of paths passing through $K$ in $\mathcal{I}^{u_1}_1$ and $\mathcal{I}^{u_2}_2$, i.e., $N_K$ in \eqref{def_RI}. We need the following lemma.
\begin{lemma}
\label{4.5}
For $d \geq 3 $, $u_1,u_2>0$ and any integer $M \geq 1$, we have
\begin{equation}
    \lim_{L_0 \to \infty}\mathbbm{P}\left[\partial_i B(0,L_0) \overset{\mathcal{I}^{u_1}_1 \cap \mathcal{I}^{u_2}_2}{\longleftrightarrow} \partial_i B(0,2L_0)|N^1_{B(0,2L_0)}=N^2_{B(0,2L_0)}=M\right]=0.
\end{equation}
\end{lemma}
\begin{proof}
Denote the probability in the left-hand side by $I$. The proof is fairly simple for $d \geq 4$. Equation \eqref{def_RI} implies that conditional on $\left\{ N^1_{B(0,2L_0)}=M\right\}$, interlacements $\mathcal{I}^{u_1}_1$ in $B(0,2L_0)$ can be seen as $M$ independent simple random walks started at some points in $\partial_i B(0,2L_0)$. The same is true for $N^2_{B(0,2L_0)}$ and $\mathcal{I}^{u_2}_2$. To connect $\partial_i B(0,L_0)$ with $\partial_i B(0,2L_0)$, at least one vertex of $\partial_i B(0,L_0)$ should be occupied. There are at most $C{L_0}^{d-1}$ many vertices but the probability of a vertex to be occupied is at most $CM^2\cdot 1/{L_0}^{d-2} \cdot 1/{L_0}^{d-2}$. Thus, 
\begin{align*}
    I &\leq C M^2 {L_0}^{d-1} \cdot \frac{1}{{L_0}^{d-2}} \cdot \frac{1}{{L_0}^{d-2}}= C \frac{1}{L_0^{d-3}}.
\end{align*}
For $d= 3$, we use Proposition \ref{key_estimate}, a more powerful estimate. By Proposition \ref{key_estimate},
\begin{align*}
    I &\leq M^2  \max_{x,y \in \partial_i B(0,2L_0)}P_{x,y}\left[X^1[0,\infty) \cap X^2[0,\infty) \cap \partial_i B(0,L_0) \neq \emptyset\right] \\
    &<  \frac{C}{\log (L_0)}.
\end{align*}
Let $L_0$ tend to $\infty$. Then, both terms above tend to 0 and we complete the proof.
\end{proof}
We define the seed event $G_x$ as the event that $\partial_i B(x,L_0)$ is connected to $\partial_i B(x,2L_0)$ in $\mathcal{K}$. Thus, $G_x$ is measurable with respect to the configuration in $B(x,2L_0)$, shift-invariant and increasing.

Let $u_1=u_2$ and $M=u_1 \cdot {L_0}^{d-2}=u_2 \cdot {L_0}^{d-2}$. Observe that there exists $B$ such that ${\rm cap}(B(0,2L_0)) \leq BL_0^{d-2}$ and $N_{B(x,2L_0)}$ $\sim$ ${\rm Poisson}(u \cdot {\rm cap} (B(0,2L_0)))$. Thus, by the Hoeffding's inequality, there exist two constants $\delta$ and $\gamma$ such that
\begin{equation*}
\begin{split}
    \mathbbm{P}\left[G_x\right] &\leq \mathbbm{P}\left[N^1_{B(x,2L_0)} > 2BM\right] + \mathbbm{P}\left[N^2_{B(x,2L_0)} > 2BM\right]\\
    & \quad +\mathbbm{P}\left[G_x | N^1_{B(x,2L_0)} \leq 2BM,N^2_{B(x,2L_0)} \leq 2BM\right]\\
    &\leq \delta e^{-\gamma M} +\mathbbm{P}\left[G_x | N^1_{B(x,2L_0)} = [2BM],N^2_{B(x,2L_0)} = [2BM]\right]
\end{split}
\end{equation*}
($[x]$ represents the largest integer $\leq x$). Take $\epsilon= 1/2$ in Proposition \ref{decoupling} and $l_0=A$ ($A$ is the integer in Proposition \ref{decoupling}). Inserting $u_1^-=u_1/2, u_1=u_1, u_2^-=u_2/2, u_2=u_2$ into \eqref{decoup_increase}, toghther with \eqref{def_epsilon} we get that
\begin{equation*}
\begin{split}
&\mathbbm{P}\left[\mathcal{K}^{\frac{1}{2}u_1,\frac{1}{2}u_2} \in G_{0,n}\right]\\ 
\leq &(2A+1)^{d \cdot 2^{n+1}} \left[\mathbbm{P}\left[\mathcal{K}^{u_1,u_2} \in G_0\right]+ \epsilon \left(\frac{1}{2}u_1, L_0, A\right)+ \epsilon \left(\frac{1}{2}u_2, L_0, A\right)\right]^{2^n}\\
\leq &(2A+1)^{d \cdot 2^{n+1}}  \Bigg{[} \delta e^{-\gamma M}+4\frac{e^{-\frac{1}{2}M{A}^\frac{d-2}{2}}}{1-e^{-\frac{1}{2}M{A}^\frac{d-2}{2}}}\\
& \quad\quad\quad\quad\quad\quad\quad\quad+ \mathbbm{P}\left[G_x | N^1_{B(x,2L_0)} = [2BM],N^2_{B(x,2L_0)} = [2BM]\right]\Bigg{]}^{2^n}.
\end{split}
\end{equation*}
Next, pick a large $M$ such that 
\begin{equation*}
    (2A+1)^{2d}\left( \delta e^{-\gamma M} + 4\frac{e^{-\frac{1}{2}M{A}^\frac{d-2}{2}}}{1-e^{-\frac{1}{2}M{A}^\frac{d-2}{2}}}\right) < \frac{1}{4}.
\end{equation*}
Finally, thanks to Lemma \ref{4.5}, we can enlarge $L_0$ such that
\begin{equation*}
    (2A+1)^{2d}\mathbbm{P}\left[G_x | N^1_{B(x,2L_0)} = [2BM],N^2_{B(x,2L_0)} = [2BM]\right]<\frac{1}{4}.
\end{equation*}
Combining the above three inequalities, we get that there exist certain $u_1=u_2>0$ and $L_0,l_0$ such that
\begin{equation}
\label{nonpercolate_123}
    \mathbbm{P}\left[\mathcal{K}^{u_1,u_2} \in G_{0,n}\right] \leq 2^{-2^n}.
\end{equation}

We claim that if 0 is connected to $\partial_i B(0,2L_n)$ in $\mathcal{K}$, then the event $G_{0,n}$ happens (see \eqref{seed} for the definition of $G_{0,n}$). We can prove this claim by induction. For $n=0$, this holds immediately. If for $n=k$ it holds, then we consider the case $n=k+1$. For a simple nearest neighbor path connecting $0$ to $\partial_i B(0,2L_{k+1})$, it must pass $\partial_i B(0,L_{k+1}/3)$ and $\partial_i B(0,2L_{k+1}/3)$. By the induction hypothesis, we can prove that the $L_k$ boxes first passed by the path in $\partial_i B(0,L_{k+1}/3)$ and $\partial_i  B(0,2L_{k+1}/3)$ satisfying $G_{x,k}$ and their distance is by definition larger than $L_{k+1}/100$. Hence, the claim holds for $n=k+1$. By induction, it holds for all $n$. So, $\mathbbm{P}\left[0 \overset{\mathcal{K}}{\longleftrightarrow} \partial_i B(0,2L_n)\right] \leq \mathbbm{P}\left[\mathcal{K} \in G_{0,n}\right]$. Together with \eqref{nonpercolate_123}, we have
\begin{equation*}
    \mathbbm{P}\left[0 \overset{\mathcal{K}}{\longleftrightarrow} \partial_i B(0,2L_n)\right]\leq 2^{-2^n}.
\end{equation*}
Let $n$ tend to $\infty$. Then, we have $\mathbbm{P}\left[0 \overset{\mathcal{K}}{\longleftrightarrow}\infty\right]=0$ for some $u_1=u_2>0$. We take $c(d)=u_1=u_2$ in Proposition \ref{nonpercolate_K} and get the non-trivial phase transition of $\mathcal{K}$.
\end{proof}
\begin{remark}
In this proof, we obtain the stretched exponential decay of connectivity function. By the method in Section 7 of \cite{soft local times}, we can greatly improve the bound to exponential decay for $d \geq 4$ and exponential decay with a logarithmic correction for $d=3$ (one can see Remark \ref{exponential decay} for more details). 
\end{remark}
We just proved that when both $u_1$ and $u_2$ tend to $0$, there are no infinite components in $\mathcal{K}$. We conjecture that for $d \geq 3$ if one fix one of $u_1$ and $u_2$ and let the other tend to $0$, there are no infinite components in $\mathcal{K}$. 

We can prove this rigorously for $d \geq 5$. We now present the proof of Proposition \ref{larger than 5}, which gives a precise statement of the claim above.
\begin{proof}[Proof of Proposition \ref{larger than 5}]
It is sufficient to prove a variant of Lemma \ref{4.5} with only one conditioning, i.e., for $d\geq 5$, $u_1,u_2>0$ and any interger $M \geq 1$
\begin{equation*}
     \lim_{L_0 \to \infty}\mathbbm{P}\left[\partial_i B(0,L_0) \overset{\mathcal{I}^{u_1}_1 \cap \mathcal{I}^{u_2}_2}{\longleftrightarrow} \partial_i B(0,2L_0)|N^1_{B(0,2L_0)}=M\right]=0.
\end{equation*}
It follows from \eqref{def_RI} that contional on $\left\{ N^1_{B(0,2L_0)}=M \right\}$, interlacements $\mathcal{I}^{u_1}_1$ in $B(0,2L_0)$ are $M$ independent simple random walks started at some random points in $\partial_i B(0,2L_0)$. Note that
\begin{equation*}
    \lim_{L_0 \to \infty} P[\mbox{no two simple random walks of these } M \mbox{ ones intersect}]=1,
\end{equation*}
for that with high probability their starting points are at least $\sqrt{L_0}$ from each other and random walks started from these points do not intersect when $d \geq 5$. Therefore, we can assume that $M=1$. By the strong Markov property, $\mathcal{I}^{u_1}_1$ in $B(0,2L_0)$ can be dominated by two independent simple random walks from some point $x$ in $\partial_i B(0,L_0)$. With high probability, these two random walks do not intersect out a small box, say $B(x,L_0/2)$. So, we only need to prove that
\begin{equation}
\label{need}
    \lim_{L_0 \to \infty}\sup_{x\in \partial_i B(0,L_0)}P_x\left[ \partial_i B(0,\frac{3}{2}L_0) \overset{ X[0, \infty) \cap \mathcal{I}^{u_2}_2}{\longleftrightarrow} \partial_i B(0,2L_0)\right]=0.
\end{equation}
The time $n$ is called a cut time if $X[0,n] \cap X(n,\infty) = \emptyset$. Call a time $n$ bad if it is a cut time and $X(i)$ is not occupied by $\mathcal{I}^{u_2}_2$, otherwise good. Given $\delta>0$, we can choose $N>0$ and $\epsilon>0$ such that for all integer $L_0 \geq 1$
\begin{equation}
\label{event1}
    P\left[\min_{i>NL_0^2}|X(i)|_{\infty}>4L_0\right]>1-\delta
\end{equation}
and
\begin{equation}
\label{event2}
    P\left[\max_{|i-j| \leq \epsilon L_0^2;i,j \leq NL_0^2}|X(i)-X(j)|_{\infty}<\frac{1}{2}L_0\right]>1-\delta.
\end{equation}
Take $N$ large first and then $\epsilon$ small. This can be proved by the central limit theorem and reflection principle. Furthermore, for arbitrary $N,\epsilon>0$ and $L_0$ large,
\begin{equation}
\label{event3}
    P\left[\mbox{there are no consecutive } \epsilon L_0^2 \mbox{ good times in }\left[0,NL_0^2\right]\right]>1-2\delta.
\end{equation}
Following is the proof of \eqref{event3}. The cut times of $X$ are independent of $\mathcal{I}^{u_2}_2$. Given $NL_0^2$ different points chronologically, for $L_0$ large 
\begin{equation*}
P\left[\mbox{some consecutive }L_0\mbox{ points are all occupied by }\mathcal{I}^{u_2}_2\right]<\delta, 
\end{equation*}
since
\begin{equation*}
    P\left[\mathcal{I}^{u_2} \mbox{ contains } M \mbox{ given points}\right] \leq Ce^{-cM^\frac{d-2}{d}}.
\end{equation*}
Therefore, 
\begin{equation*}
P\left[\mbox{there is a bad time in every consecutive }L_0 \mbox{ cut times}\right]>1- \delta.
\end{equation*}
In addition, for a random walk in $d \geq 5$, the density of cut times will converge to a positive constant as shown in (1) of \cite{cut time}. Combining these two facts, we can get \eqref{event3}.

If the above three events in \eqref{event1}, \eqref{event2} and \eqref{event3} occur at the same time, then the event in \eqref{need} will not happen for that $X$ leaves $B(0,2L_0)$ after $NL_0^2$ and the first $NL_0^2$ steps are cut into disjoint subpaths with diameter smaller than $L_0/2$ by the bad times. Let $\delta$ tend to zero and the proof of \eqref{need} is completed. This completes the proof for the case $d\geq 5$. 
\end{proof}
\begin{remark}
For $d=4$, the bi-infinite simple random walk still has cut times and bad times, so we guess that $\partial_i B(0,L_0)$ and $\partial_i B(0,2L_0)$ can be disconnected by bad times. However, for $d=3$, there are no cut times for a doubly-infinite simple random walk, i.e., $\zeta_3>1/4$ (see (2) to (3) in \cite{cut time}), and a new approach is required.
\end{remark}
\section{On the coexistence of infinite clusters}
\label{section_5}
In this section, we will consider the phase diagram of $\mathcal{K}$ and $\mathcal{V}$ put together and prove Theorem \ref{1.4} which is split into two parts, namely two propositions below.

\begin{proposition}
\label{5.1}
There exists a region in the phase diagram such that $\mathcal{K}$ and $\mathcal{V}$ both have a unique infinite component a.s.
\end{proposition}
\begin{proof}
Take $u_1\in (0,u_*)$ and $u_2$ sufficiently large. By Propositions \ref{existence_V} and \ref{percolate_K}, we conclude that both $\mathcal{K}$ and $\mathcal{V}$ have an infinite component.
\end{proof}

Next, we consider whether there is some region such that neither of $\mathcal{K}$ and $\mathcal{V}$ percolates, or equivalently the occupied vertices and the vacant vertices wrap each other. In general, this problem is difficult in low dimensions except $d=2$ due to lack of adequate tools. We claim that in high dimensions there does not exist such region, see the following proposition for a rigorous statement.
\begin{proposition}
\label{5.2}
There exists $D_2<\infty$ such that for all $d>D_2$ and $u_1,u_2>0$ at least one of $\mathcal{K}$ and $\mathcal{V}$ has a unique infinite component a.s. 
\end{proposition}
\begin{proof}
The proof relies on Theorem 1.3 and Theorem 0.1 of \cite{lower bound}. By Theorem 0.1 in \cite{lower bound}, we know that 
\begin{equation*}
    \liminf_{d\rightarrow \infty}\frac{u_*(d)}{\log(d)} \geq 1.
\end{equation*}
Thus, there exists a constant $D$ such that for all $d \geq D$, we have $u_*(d) \geq  \log (d) /2$. Theorem \ref{1.3} tells us that when $d \geq D_1$ and $u_1,u_2>1$, the intersection $\mathcal{K}$ percolates. Let $D_2=\min \{D_1,D,10 \}$.

For all $d>D_2$ and $u_1,u_2>0$, there are two cases.

(1).\ $\min \{ u_1,u_2 \} <u_*(d)$. By Proposition \ref{existence_V}, the vacant set $\mathcal{V}$ percolates.

(2).\ $\min \{ u_1,u_2 \} \geq u_*(d)>1$. By Theorem \ref{1.3} and $d>D_2 \geq D_1$, the intersection $\mathcal{K}$ percolates.
\end{proof}
The following part is devoted to the proof of Theorem \ref{1.3}. The proof contains two parts. The first is to prove that in each hypercube $\{ 0,1 \}^d$, with high probability, there is a ubiquitous connected component (meaning that most vertices in the hypercube are connected to it, see below for a rigorous definition) and this ubiquitous component is also connected to those of the neighboring hypercubes (which also exist with high probability). The second step is to prove that such hypercubes with ubiquitous components percolate in the whole space. $H$ represents a hypercube. For $x \in \mathbb{Z}^d$, let $H_x =x+ \{ 0,1 \}^d$ be the hypercube at $x$. Recall that Bernoulli site percolation with parameter $p$ is a model in which each vertex is independently occupied with probability $p$ and vacant with probability $1-p$, denoted by Bernoulli$(p)$.
\begin{proof}[Proof of Theorem \ref{1.3}]
Lemma \ref{dominate} says that $P_0 ({\tau}_{\{ 0,1\} ^d}^+=\infty )>c_1$. Thanks to this Lemma, we claim that the interlacements with intensity $1$ in a hypercube $H$ can dominate Bernoulli site percolation with parameter $1-e^{-c_1^2}$. For a vertex $x$ in $H$, we only consider the paths that pass through $H$ only at $x$ and only once. They have intensity $P_x({\tau}_{H}^+=\infty) \cdot P_x({\tau}_{H}^+=\infty) \geq c_1^2$. Hence, conditional on all the other vertices in $H$, the probability that $x$ is occupied is at least $1-e^{-c_1^2}$. This completes the proof of the claim. Furthermore, $\mathcal{K}^{1,1}$ can dominate Bernoulli site percolation with parameter $(1-e^{-c_1^2})^2$ in $H$. Let $3p=(1-e^{-c_1^2})^2$. It is sufficient to prove the increasing properties about $\mathcal{K}^{1,1}$ in $H$ for Bernoulli site percolation with parameter $3p$.

Next, we will define some notation to be used later. We want to mention that all the inequalities below hold when $d$ is larger than a universal constant and all the constants are independent of $d$. Let $n=2^d=|H|$. For any subset $X$ of a hypercube $H$, write $N(X)$ for all the neighbors of $X$ in $H$ and $\overline{X}$ for $X \cup N(X)$. A connected component of $H$ is called an atom if it contains more than $d^{100}$ vertices. Call a connected component $A$ of $H$ a ubiquitous component if $\left|\overline{A}\right|>\left(1-1/d^2\right)n$. There is only one ubiquitous component in a hypercube $H$. Suppose that in contrast $H$ has two distinct ubiquitous components $A$ and $B$. Since $B \cap \overline{A} = \emptyset$, we have $|B| \leq n-\left|\overline{A}\right| <n/d^2$. Hence,
\begin{align*}
    |\overline{B}|& = |B|+|N(B)| \\
    &\leq |B| +d|B| < \left(1-1/d^2\right)n.
\end{align*}
$B$ is not a ubiquitous component. Therefore, $H$ has at most one ubiquitous component.

The seed event $G_x$ is defined as the intersection of the following two events: (1).\ for any $e \in \{(0,0),(0,1),(0,-1),(1,0),(-1,0) \} \times \{0 \}^{d-2}$, $\mathcal{K}^{1,1} \cap H_{x+e}$ has a ubiquitous component; (2).\ all the above five ubiquitous components are connected in $\mathcal{K}^{1,1}$. The event $G_x$ is measurable with respect to the configuration in $B(x,1)$, shift-invariant and increasing. $\overline{G}_x$ is the complement of $G_x$. We first prove the following inequality: 
\begin{equation}
\label{highdimen_local}
    \lim_{d \rightarrow \infty}d^3\mathbbm{P}\left[\overline{G}_x\right]=0.
\end{equation}

It is sufficient to prove this property for Bernoulli site percolation with parameter $3p$. Now, $H_x$ is a fixed hypercube and $H_1, H_2, H_3, H_4$ are the four neighboring hypercubes of $H_x$ in the first and second directions. The proof follows three steps. With high probability, (1).\ in Bernoulli(p), most vertices in $H_x$ have a neighbor in an atom; (2).\ in Bernoulli(2p), these atoms are connected to a ubiquitous component in $H_x$; (3).\ in Bernoulli(3p), this ubiquitous component is connected with those of the neighboring four hypercubes. If the above three events happen simultaneously, then $G_x$ happens. 

Consider Bernoulli site percolation on $H_x$ with parameter $p$. For a fixed vertex of $H_x$, we can construct a 1000-high tree in $H_x$ in which every node except leaves has more than $d/2000$ descendants. By the Hoeffding's inequality, $P[Y] \geq 1-e^{-Cd}$, where $Y$ represents the event that each node except leaves has more than $pd/10000$ occupied descendants. When $Y$ happens, this vertex has a neighbor in an atom. Thus, by the Markov's inequality
\begin{equation*}
    P[Z] \geq 1-e^{-Cd},
\end{equation*}
where $Z$ represents the event that except $n/e^{Cd}$ vertices, every vertex of $H_x$ has a neighbor in an atom.

Consider, now, the set of atoms obtained in $H_x$. We open all the vacant vertices independently with probability $q=p/(1-p)$. With these additional open vertices, the atoms in $H_x$ have a large probability to be connected to a ubiquitous component. We claim that with high probability no union of atoms $A$ covering more than $n/d^5$ vertices can be separated from the union of all the other atoms $B$, when $B$ also has at least $n/d^5$ vertices. This follows from the following lemma.
\begin{lemma}
Assume that $Z$ happens and $A, B \subset H_x$ satisfy the above conditions including $|A|,|B|>n/d^5$ and $A \cap \overline{B} = \overline{A} \cap B= \emptyset$. Then, there exist $cn/d^7$ pairwise disjoint paths connecting $A$ and $B$, which have length at most three, i.e., in the form $A\leftrightarrow y \leftrightarrow B$ or $A \leftrightarrow y \leftrightarrow z \leftrightarrow B$.
\end{lemma}
\begin{proof}
There are two cases. Case $1$: $|N(A) \cap N(B)| >n/d^7$, then take $A\leftrightarrow y \leftrightarrow B$, where $y \in N(A) \cap N(B)$. Case $2$: $|N(A) \cap N(B)| <n/d^7$. Suppose that $\left|\overline{A}\right|\leq \left|\overline{B}\right|$, then we can prove that $\left|\overline{A}\right|<3n/4$. With the isoperimetric inequality, $\left|N\left(\overline{A}\right)\right|>\left|\overline{A}\right|\left(n-\left|\overline{A}\right|\right)/(nd)>cn/d^6$. In addition, by $Z$, all the vertices in $N\left(\overline{A}\right)$ except $n/e^{Cd}$ ones should have a neighbor in $B$. Thus, there are at least $cn/d^6$ different paths in the form $A \leftrightarrow y  \leftrightarrow B$ or $A \leftrightarrow y \leftrightarrow z  \leftrightarrow B$, where $y \in N(A)$ and $ z\in N\left(\overline{A}\right)$. Since $y$ can be in at most $d$ different paths, there are at least $cn/d^7$ disjoint paths.
\end{proof}
The number of choices of $A$ and $B$ satisfying the above conditions is at most $2^{n/d^{100}}$. By the above lemma, each pair has smaller than ${(1-q^2)}^{cn/d^7}$ probability to be still disconnected. So, the probability of existing such $A$ and $B$ is at most ${(1-q^2)}^{cn/d^7} \cdot 2^{n/d^{100}}$. If there are no such pairs of $A$ and $B$, we obtain a ubiquitous component (because the vertices not in this component is at most $n/e^{Cd}+d n/d^5$). Therefore,
\begin{equation*}
P[U] \leq  e^{-Cd} +{\left(1-q^2\right)}^{cn/d^7} 2^{n/d^{100}},
\end{equation*}
where $U$ represents the event that there are no ubiquitous components in $H_x$. By symmetry,
\begin{equation*}
P[V] \leq  5\left(e^{-Cd} +{\left(1-q^2\right)}^{cn/d^7} 2^{n/d^{100}} \right),
\end{equation*}
where $V$ represents the event that there are no ubiquitous components in $H_1$, $H_2$, $H_3$, $H_4$ or $H_x$.

Suppose that $V$ does not happen and we obtain five ubiquitous components in these five hypercubes. The final step is to connect the ubiquitous component in $H_x$ with the neighboring four ones. We open all the vacant vertices independently with probability $r=p/(1-2p)$. Note that there are $2^{d-1}$ common vertices for the neighboring two hypercubes $H_x$ and $H_1$ and at least $2^{d-1}-2n/d^2>cn$ of them are connected to both the two ubiquitous components in $H_x$ and $H_1$. Thus, with probability more than $1-(1-r)^{cn}$, the two ubiquitous components in $H_x$ and $H_1$ are connected with each other, the same for $H_2$, $H_3$ and $H_4$.

Summing over all the probabilities of bad events, we have
\begin{equation*}
    P\left[\overline{G}_x\right] \leq 5 \left( e^{-Cd} +{\left(1-q^2\right)}^{cn/d^7}\cdot 2^{n/d^{100}} \right)+ 4 (1-r)^{cn}.
\end{equation*}
This completes the proof of \eqref{highdimen_local}.

\begin{remark}
The most costly step is the first one. It is easy to find that this argument also holds when $u=\frac{1}{d^{1/2-\epsilon}}$ for any $\epsilon >0$.
\end{remark}

For the second part, we directly use Theorem 2.2 in \cite{lower bound}, i.e.,
\begin{proposition}
$G_x$ is measurable with respect to the configuration in $B(x,1)$ and shift-invariant. If $\limsup_{d}d^3\mathbbm{P}\left[\overline{G}_x\right]<\infty$, then $G_x$ will percolate in the slab $\mathbb{Z}^2 \times \{ 0 \} ^{d-2}$ for some large $d$ meaning that there exists an infinite long nearest neighbor path in which every vertex satisfies $G_x$.
\end{proposition}
The proof of this result uses a direct decoupling inequality rather than the sprinkling version, so we do not even need monotonicity of the seed event. Observe that $u \leq d$ also holds here.

Once there exists an infinitely long nearest neighbor path in which every vertex satisfies $G_x$, there is an infinite component of $\mathcal{K}$ along this path, because every hypercube along this path has a ubiquitous component and all these components are connected with the neighboring four ones.
\end{proof}

\renewcommand{\abstractname}{Acknowledgements}
\begin{abstract}
The author wishes to thank Xinyi Li for suggesting the problem and for useful discussions and a careful reading of the manuscript.
\end{abstract}

\end{document}